\documentclass[11pt]{article}

\usepackage[utf8]{inputenc}
\usepackage{scalerel}
\usepackage{enumitem}
\usepackage{authblk}
\usepackage{wrapfig}
\usepackage{tikz,tikz-cd}
\usepackage{amsmath,amscd,amssymb}
\usepackage{amsthm}
\usepackage{amsfonts}
\usepackage{mathtools}
\usepackage{color}
\usepackage[mathscr]{eucal}
\usepackage[margin=1in]{geometry}
\usepackage{subcaption}
\usepackage{overpic}
\usepackage{url}
\usepackage{extarrows}
\usepackage{booktabs} 

\setlist[enumerate]{itemsep=-1mm}

\usepackage[multiple]{footmisc} % Multiple footnotes at one point

 \newtheorem{theorem}{Theorem}[section]
 \newtheorem{proposition}[theorem]{Proposition}
 
 \newtheorem{lemma}[theorem]{Lemma}

 \theoremstyle{definition}
 \newtheorem{definition}[theorem]{Definition}
 \newtheorem{example}[theorem]{Example}
 \newtheorem{remark}[theorem]{Remark}
 
 \newtheorem{algorithm}[theorem]{Algorithm}

\newcommand{\real}{\ensuremath{\mathbb{R}}}
\newcommand{\orient}[1]{\ensuremath{\hat{#1}}}
\newcommand{\inner}[2]{\langle #1,#2 \rangle}

\newcommand{\innervf}[2]{\langle\!\langle #1,#2 \rangle\!\rangle}

\newcommand{\vf}{\ensuremath{\mathfrak{X}}}
\newcommand{\borel}[1]{\ensuremath{\mathcal{P}(#1)}}
\newcommand{\Ric}{\ensuremath{\mathrm{Ric}}}

\begin{document}

%\linenumbers

\title{Covariance and Principal Component Analysis on Riemannian Manifolds and Graphs}
\author[1]{Luiz Hartmann}
\author[2]{Wenwen Li}
\author[2]{Washington Mio}
\affil[1]{Departamento de Matem\'{a}tica, Universidade Federal de S\~{a}o Carlos, Brazil}
\affil[2]{Department of Mathematics, Florida State University,  USA}

\date{ }
\maketitle

\begin{abstract}
We develop notions of covariance and principal component
analysis (PCA) for probability measures on Riemannian manifolds of bounded geometry and a discrete counterpart for distributions on the vertex sets of weighted simple graphs. Rather than anchoring variation at the Fr\'{e}chet mean, whose uniqueness and usefulness can fail in this context, we consider variation about every point. This leads to a representation of each point on the manifold by a vector field derived from the heat kernel and thus to a map of the underlying manifold into a Hilbert space of vector fields, in which covariance and PCA carry over from the Euclidean setting. A distinguishing feature of this formulation is that one typically obtains infinitely many principal components, capable of capturing highly nonlinear geometric features and patterns of variation. We prove an embedding theorem for the mapping into the space of vector fields and develop two computational reductions of the theory: RiePCA, a finite-dimensional reduction based on vector fields supported on finitely many points, and GraphPCA, a discrete formulation for measures on weighted graphs in which vector fields assign orientations and magnitudes to edges. Numerical examples and experiments illustrate the behavior of the proposed methods on both synthetic and real data.

\bigskip
\noindent
{\em Keywords:} Covariance on manifolds, covariance on graphs, principal component analysis, principal vector fields, dimension reduction, manifold learning.

\medskip
\noindent
{\em 2020 Mathematics Subject Classification:} 62R30, 62H25, 62R07
\end{abstract}

\section{Introduction} \label{S:intro}

The mean and the covariance tensor of a probability measure on Euclidean space $\mathbb{R}^d$ are basic summary statistics underpinning widely used techniques for data visualization, analysis, and inference such as Principal Component Analysis (PCA). This paper develops the concepts of covariance tensor and PCA for probability distributions on Riemannian manifolds and weighted simple graphs. 

Let $(M,g)$ be a $d$-dimensional Riemannian manifold. Here, $g$ denotes the Riemannian structure, a smoothly varying family of inner products $g_x$ on each tangent space $T_x M$ for $x \in M$. As detailed below, in our approach, each point $y \in M$ is represented by a vector field $v_y$ on $M$ and covariance is formulated as a 2-tensor on a Hilbert space $\vf_2(M)$ of vector fields, an infinite-dimensional space if $M$ has dimension $d \geq 1$. The resulting principal modes of variation are thus given by principal vector fields. Unlike standard Euclidean PCA, typically there are infinitely many mutually orthogonal principal components that can uncover highly nonlinear geometric features in data. 

PCA for a probability measure $\nu$ on $\mathbb{R}^d$ (with finite second moments) is founded on the covariance tensor:
\begin{equation} \label{E:cov1}
\Sigma = \int_{\mathbb{R}^d} (y-x_0) \otimes (y-x_0) \, d\nu(y),
\end{equation}
where $x_0 \in \real^d$ denotes the mean of $\nu$. Although there has been considerable interest in defining covariance and extending PCA to distributions on Riemannian manifolds, significant challenges remain related to the non-uniqueness of means and the lack of a global linear structure on $M$. Moreover, a fundamental limitation of standard Euclidean PCA is that, despite the infinite-dimensional nature of the space of probability distributions on $\mathbb{R}^d$, the method yields at most $d$ components. Typically, the number of structurally meaningful components is much smaller ($\ll d$). To overcome these limitations, an infinite-dimensional representation offers a more natural and powerful framework. While Kernel PCA \cite{Scholkopf1998} represents a step toward addressing these issues, it often involves a designer kernel function that can be extrinsic to the geometry of the underlying space. Although ultimately our approach can be viewed as a particular form of kernel PCA, it directly addresses the absence of global linearity, the lack of natural reference points such as means, and infinite dimensionality. Moreover, the embedding of $M$ into the space of vector fields is explicit and derived from the Laplace-Beltrami operator, which is intrinsic to $(M,g)$.

The notion of mean of a distribution can be extended to Riemannian manifolds and other metric spaces through the Fr\'{e}chet mean (cf.\,\cite{frechet1948}). Nonetheless, the Fr\'{e}chet mean may be far from unique, and even when uniqueness holds, it is in general unclear how useful the mean is as a reference point for measuring variation; for example, if the distribution is multimodal or has highly nonlinear support. Another obstacle to extending covariance tensors to manifolds is the seemingly essential use of the vector space structure of $\real^d$ in the random ``displacement from the mean'' vector $(y-x_0)$ that appears in the integrand in \eqref{E:cov1}. On a Riemannian manifold, only local linearity is available at each $x \in M$ through the tangent space $T_x M$. To circumvent this problem, we reinterpret the random vector $(y-x_0)$ as follows.  Consider the function $u(x,y) = \|y-x\|^2/2$ that we may think of as the potential energy of $y$ relative to $x$. Then, $-\nabla_x u (x,y) = y-x$ so that we can rewrite \eqref{E:cov1} as
\begin{equation} \label{E:cov3}
\Sigma = \int_{\real^d} \nabla_x u(x_0,y) \otimes  \nabla_x u(x_0,y) \, d\nu(y),
\end{equation}
an expression that only invokes local linearity and more easily generalizes to the Riemannian setting. To express \eqref{E:cov3} in the language of vector fields, for each $y \in \real^d$, define the vector field $v_y (x) \coloneqq \nabla_x u(x,y)= x-y$. Then $v_y(x_0) = x_0-y$ and the mean vector field
\begin{equation}
\bar{v} (x) = \int_{\real^d} v_y (x) \, d\nu(y) = x-x_0   
\end{equation}
has the property that $\bar{v} (x_0) =0$. Thus, we can write the covariance tensor in \eqref{E:cov3} as
\begin{equation} \label{E:cov4}
\Sigma = \int_{\real^d} (v_y (x_0) - \bar{v} (x_0)) \otimes  
(v_y (x_0) - \bar{v} (x_0))\, d\nu(y),
\end{equation}
an expression of the covariance in terms of vector fields evaluated at the mean $x_0$. Since on a Riemannian manifold there is generally no natural choice of a reference point such as $x_0$, we formulate covariance using the full vector-field analogs of $v_y$. 

The potential function $u(x,y) = d^2_g (x,y)/2$, where $d_g$ is the geodesic distance in $(M,g)$, may seem like a good candidate for the formulation of Riemannian covariance. However, this $u$ has some of the same limitations as the Euclidean potential, for example, in its inability to capture multi-modality. Additionally, the derivative $\nabla_x u(x,y)$ may not exist if $x$ falls in the cut locus of $y$. This motivates a multi-scale approach based on the heat kernel $k_t \colon M \times M \to \real$, $t>0$, a smooth potential that produces a windowing effect and is thus sensitive to structural features at different scales. At scale $t>0$, points on $M$ are represented as vector fields through the embedding $\Phi_t \colon M \to \vf_2 (M)$ given by $\Phi_t(y)= v_y$, where $v_y$ is the vector field
\begin{equation}
v_y (x)\coloneqq \nabla_x k_t (x,y).    
\end{equation}
In this formulation, the covariance tensor of a probability measure $\mu$ is defined as
\begin{equation} \label{E:covintro}
\Sigma_\mu \coloneqq \int_M (v_y-\bar{v}) \otimes (v_y-\bar{v}) \,d\mu (y),    
\end{equation}
where $\bar{v}$ is the mean vector field $\bar{v} = \int_M v_y \,d\mu (y)$. The covariance operator associated with $\Sigma_\mu$ is a self-adjoint, positive semi-definite, trace-class operator so that PCA in this setting can be formulated as in the Euclidean case. We note that for a fixed $y \in M$, unlike the Euclidean potential $u(x,y)= \|x-y\|^2/2$, which forms an energy well about $y$, the heat kernel is an energy hill, so it might seem more natural to work with the negative gradient field. However, as a negative sign has no effect on the covariance as defined in \eqref{E:covintro}, we work directly with the heat kernel.

The formulation of manifold covariance and PCA discussed above serves as the basis for two discrete models.  The covariance tensor defined in \eqref{E:covintro} is designed to account for covariation of a distribution $\mu$ about all points $x \in M$ because, in general, there may be no particularly salient reference points. In some situations, however, there is a set of prominent points such as finitely many isolated local maxima of the heat-kernel density estimator for the distribution $\mu$. In such scenarios, we introduce a variant of $\Sigma_\mu$ that is based on vector fields supported on a fixed finite set of anchor points. Targeting applications, this model is developed for potential functions more general than the heat kernel. We refer to this finite-dimensional version of Riemannian PCA as RiePCA. 

We also develop a discrete formulation of covariance and PCA, termed GraphPCA, for probability measures supported on the vertex set of a weighted finite simple graph. In this setting, a vector field is modeled as an assignment of an orientation and magnitude to each edge of the graph. This graph model is also useful in the analysis of manifold data if the ground manifold $M$ is high-dimensional (or even infinite-dimensional), as discretizing a vector field over a fine, dense grid is not feasible. One can instead construct a graph (such as a nearest-neighbors graph) from the data and uncover modes of variation intrinsic to the dataset. 

Several approaches have been proposed for covariance and PCA on Riemannian manifolds. Principal Geodesic Analysis (PGA)~\cite{Fletcher2004} is a pioneering method that adapts Euclidean PCA to a Riemannian manifold $M$. It projects data from $M$ onto the tangent space at the Fr\'echet mean (if uniquely defined) via the Riemannian logarithm map, where standard PCA is performed. The resulting principal components are then interpreted on the manifold $M$ via the exponential map. A probabilistic framework for principal geodesic analysis has been developed in~\cite{Zhang2013}. Geodesic PCA (GPCA)~\cite{Huckemann2010,Huckemann2006} offers a more intrinsic approach that bypasses tangent-space linearization, directly optimizing data projection onto geodesics to minimize reconstruction errors. A formulation of covariance that uses potential functions and measures covariation about all points of $M$ has been outlined in \cite{hang2018}; however, a key difference from the present method is that covariance is defined pointwise and assembled as a covariance field, not as a single global tensor at the level of vector fields. The approach developed in \cite{AbuqraisPigoli2026} also accounts for correlations about points other than the Fr\'{e}chet mean using tangent-space linearization via the Riemannian logarithm. Methods such as principal nested  spheres~\cite{Jung2012} employ a form of backward PCA that applies to data on spheres, whereas the barycentric subspaces method of~\cite{Pennec2018} applies analogous principles to more general Riemannian manifolds. For probability distributions on Euclidean space, a multi-scale model has been developed in \cite{diazmartinez2020}, but covariance is also defined pointwise, not as a single global tensor, making it difficult to formulate global PCA.

For data with a graph structure, other approaches to covariance and PCA have been proposed. 
Saerens et al.~introduced PCA on graphs based on Euclidean commute-time distance (ECTD)~\cite{saerens2004principal}. Their construction can be interpreted as (unweighted) kernel PCA with the pseudoinverse of the graph Laplacian as its kernel. In contrast, our method incorporates a probability measure $\mu$ on the vertex set. 
Notions of variance and covariance for distributions on the vertex set of a weighted graph are also defined in~\cite{devriendt2022variance}. However, covariance is defined for a pair of distributions on a graph to measure how two network distributions or processes co-vary with respect to network geometry and topology, in sharp contrast with our approach that is more in line with the usual notion of covariance of a distribution and leads up to GraphPCA that uncovers the principal modes of variation of a single distribution. Observable Component Analysis is a form of PCA on metric spaces developed in \cite{karacam2025} that is based on covariation of 1-Lipschitz reductions to the real line. Some key features that distinguish the approach based on observables from the methods of this paper are that our model is more intrinsic to the underlying manifold or graph, and more amenable to computation.

\paragraph{Organization.} Section \ref{S:prelim} reviews some basic properties of the heat kernel and Section \ref{S:cvt} introduces and discusses the properties of the Riemannian covariance tensor. PCA on manifolds is developed in Section \ref{S:rpca}, Section \ref{S:gfembed} proves an embedding theorem for closed manifolds into the Hilbert space of $L^2$ vector fields, and Section \ref{S:kpca} interprets manifold PCA as kernel PCA. A finite-dimensional reduction of the Riemannian model, termed RiePCA, is presented in Section \ref{S:findim}. Graph covariance and GraphPCA are the main themes of Section \ref{S:graphpca}, whereas Section \ref{S:numerics} presents the results of some numerical experiments. We close the paper with a summary and discussion in Section \ref{S:summary}. An embedding theorem for non-compact manifolds of bounded geometry is proven in the Appendix.

%-----------------------

\section{Preliminaries} \label{S:prelim}

Let $(M,g)$ be a $d$-dimensional Riemannian manifold, $d \geq 1$. We denote the {\em geodesic distance} on $M$ by $d_g$ and the {\em volume measure} by $vol$. We write $TM$ for the tangent bundle of $M$, $T_x M$ for the tangent space at $x \in M$, $\inner{\ }{\, }_x$ for the inner product on $T_x M$ given by the Riemannian structure, and $\|\cdot\|_x$ for the norm associated with $\inner{\ }{\, }_x$, that is, $\|h\|_x^2 = \inner{h}{h}_x$, for any $h \in T_x M$. 

A {\em vector field} on $M$ is a section of the tangent bundle $TM$; in other words, a mapping $v \colon M \to TM$ such that $v(x) \in T_x M$, for any $x \in M$. We denote by $\vf_2 (M)$ the vector space of all vector fields satisfying $\int_M \|v(x)\|^2_x dvol (x) < \infty$, after identifying vector fields that only differ on a set of volume zero. $\vf_2 (M)$ is a separable Hilbert space with the inner product
\begin{equation} \label{E:innervf}
\innervf{v}{w} \coloneqq \int_M \inner{v(x)}{w(x)}_x dvol (x).
\end{equation}
We also view the tensor product $\vf_2(M) \otimes \vf_2(M)$ as a Hilbert space with the inner product given on pure tensors by
\begin{equation}
\begin{split}
\innervf{v_1 \otimes v_2}{w_1 \otimes w_2}_\otimes 
&\coloneqq \innervf{v_1}{w_1} \cdot \innervf{v_2}{w_2} \\
&\,=\int_M \inner{v_1(x)}{w_1(x)}_x dvol (x) \cdot \int_M \inner{v_2(x)}{w_2(x)}_x dvol (x),
\end{split}
\end{equation}
extended by linearity to their linear span and via limits to the completion. The corresponding norm is denoted $\|\cdot\|_\otimes$.

Unless stated otherwise, we assume throughout that $M$ is a connected, complete $d$-dimensional Riemannian manifold of bounded geometry, $d \geq 1$. Bounded geometry means that the curvature of $M$ and its covariant derivatives of all orders are uniformly bounded and the (global) injectivity radius $inj(M)$ is strictly positive. More precisely, 
\begin{enumerate}[label=\rm{(\roman*)}]
\item for each integer $k\geq 0$, there exists a constant $C_k>0$ such that $\|\nabla^k R\| \leq C_k$, where $R$ denotes the Riemannian curvature tensor of $M$;
\item there exists $\imath_0 >0$ such that $inj (M) \geq \imath_0$.
\end{enumerate}
Manifolds of bounded geometry include many of the manifolds of interest in data analysis such as compact Riemannian manifolds, Euclidean spaces, and hyperbolic spaces. The bounds on curvature imply that the Ricci curvature of $M$, denoted $\Ric_M$, is bounded below, a fact that is used repeatedly in our study of Riemannian covariance.

Next, we review some basic facts about the heat kernel. The {\em heat equation} on $M$ is given by
\begin{equation} \label{E:hequation}
\frac{\partial u}{\partial t} = \Delta u,
\end{equation}
where $\Delta$ is the Laplace-Beltrami operator $\Delta= \text{div} \circ \nabla$. Given a continuous function $f \colon M \to \real$ with compact support, the Cauchy problem consists of finding $u \colon M \times [0,\infty) \to \real$ that satisfies \eqref{E:hequation} on $M \times (0,\infty)$ and $u(x,0) = f(x)$, $\forall x \in M$. The {\em heat semigroup} $e^{t \Delta}$, $t>0$, is a family of integral operators that generate the solutions $e^{t\Delta} f$ to the Cauchy problem. The {\em heat kernel} is a smooth function $k \colon M \times M \times (0,\infty) \to \real$ such that
\begin{equation}
e^{t\Delta} f (x) = \int_M k_t (x,y) f(y) dvol(y),  
\end{equation}
where $k_t (x,y) = k(x,y,t)$. The heat kernel satisfies the following properties:
%----------------
\begin{enumerate}[label=\textup{(\roman*)}]
\item (positivity) $k_t (x,y)>0$, for all $x,y \in M$ and $t>0$ since $M$ is connected;
\item (symmetry) $k_t (x,y)=k_t(y,x)$, $\forall x,y \in M$ and $\forall t>0$;
\item $\int_M k_t (x,y) dvol(x)=1$, for all $y\in M$ and $t>0$ because $M$ has bounded geometry;
\item (semigroup property) $k_{t+s} (x,y) = \int_M k_t(x,z) k_s(z,y) dvol (z)$, for all $x,y \in M$ and $s,t>0$;
\item if $f\colon M \to \real$ is a smooth function with compact support, then
\begin{equation}
\int_M k_t (\cdot,y) f(y) dvol (y) \to f(\cdot),
\end{equation}
as $t \to 0^+$, where the convergence is in $C^\infty (M)$.
\end{enumerate}

\begin{example} \label{EX:rd}
In Euclidean space $\real^d$, $d \geq 1$, the heat kernel is given by
\begin{equation}
k_t(x,y)=\frac{1}{(4\pi t)^{\frac{d}{2}}} \exp
\left(-\frac{\|x-y\|^2}{4t}\right),
\end{equation}  
the Gaussian kernel with variance $\sigma_t^2 = 2t$.
\end{example}

If $M$ is a closed manifold, then $L^2(M)$ (with respect to the volume measure) admits a complete orthonormal set $\{\phi_i\colon M \to \real \colon i \geq 0\}$ of eigenfunctions of $\Delta$ such that $\Delta \phi_i = -\lambda_i \phi_i$  with $\lambda_i \geq 0$. We always assume that the eigenfunctions have been indexed so that the sequence $\lambda_i$, $i \geq 0$, is non-decreasing. The sequence of eigenvalues satisfies
\begin{equation}
\lim_{i \to \infty} \lambda_i = \infty.
\end{equation}
Since $M$ is connected, $\lambda_0 =0$ is the only trivial eigenvalue of $\Delta$. Moreover, the heat kernel admits the spectral decomposition
\begin{equation}
k_t (x,y) = \sum_{i=0}^\infty e^{-\lambda_i t} \phi_i(x) \phi_i (y).    
\end{equation}

%----------------------

\section{Riemannian Covariance} \label{S:cvt}

We denote by $\borel{M}$ the collection of all Borel probability measures on $M$.

\begin{definition} \label{D:cvt}
Let $\mu \in \borel{M}$ and $t>0$. For each $y \in M$, let $v_y$ be the vector field given by 
\[
v_y (x) = \nabla_x k_t(x,y)
\]
and $\bar{v}=\int_{M}  v_y\, d\mu(y)$ be the mean vector field (with respect to $\mu$). The {\em covariance tensor} $\Sigma_\mu \in \vf_2(M) \otimes \vf_2(M)$ at scale $t>0$ is defined as
\[
\Sigma_\mu \coloneqq \int_M (v_y-\bar{v}) \otimes (v_y-\bar{v}) \,d\mu (y),
\]
the expected value of the random 2-tensor $(v_y-\bar{v}) \otimes (v_y-\bar{v})$. (The vector fields $v_y$ and the covariance tensor $\Sigma_\mu$ depend on the parameter $t>0$ but we omit this dependence from the notation as $t$ is fixed throughout.)
\end{definition}

 Note that, in this formulation, the covariance tensor is defined for all $\mu \in \borel{M}$, without the requirement of $\mu$ having finite second moments. We elaborate on this point in Section \ref{S:rpca}. We now verify that $\Sigma_\mu \in \vf_2(M) \otimes \vf_2(M)$. Since $M$ has bounded geometry, there exists $K\geq 0$ such that the Ricci curvature satisfies $\Ric_M \geq -K (d-1)$, where $d \geq 1$ is the dimension of $M$. Analogous to the Li-Yau estimates for the heat kernel \cite{LY}, Davies \cite{DAV1} showed that the gradient of $k_t$ satisfies
\begin{equation}\label{E:gradhk}
	\|\nabla_x k_t(x,y)\|_x \leq C_1(t,M)\cdot e^{-\frac{d_g^2(x,y)}{c_t}},
\end{equation}
where $C_1(t,M)>0$ and $c_t>0$ are constants independent of $x$ and $y$. Estimate \eqref{E:gradhk} ensures that
\begin{equation} 
\|v_y\|^2 = \int_M \|\nabla_x k_t(x,y)\|^2_x \,dvol(x) \leq C_1^2(t,M) \int_M e^{-\frac{d_g^2 (x,y)}{c_t}} dvol(x),
\end{equation}
$\forall y \in M$. By Lemma \ref{L:convergence}, there is a constant $C>0$ such that $\int_M e^{-\frac{d_g^2 (x,y)}{c_t}} dvol(x)<C$, for all $y \in M$. Therefore, 
\begin{equation}\label{E:est2}
\|v_y\|^2 = \int_M \|\nabla_x k_t(x,y)\|^2_x \,dvol(x) \leq C \cdot C_1^2(t,M) \eqqcolon D(t,M).
\end{equation}
As such, the mean field $\bar{v}$ also satisfies
\begin{equation} \label{E:est3}
\|\bar{v}\|^2 =\|\int_M v_y  \,d\mu (y)\|^2
\leq \int_M \|v_y\|^2  d\mu (y) \leq D(t,M).
\end{equation}
By \eqref{E:est2} and \eqref{E:est3},
\begin{equation} \label{E:bound}
\|v_y-\bar{v}\|^2 \leq \|v_y\|^2 + 2\|v_y\|\|\bar{v}\| + \|\bar{v}\|^2
\leq 2\|v_y\|^2 + 2\|\bar{v}\|^2  \leq 4D(t,M).
\end{equation}
This implies that
\begin{equation} \label{E:est4}
\|\Sigma_\mu\|_\otimes  \leq \int_M \|(v_y-\bar{v}) \otimes (v_y-\bar{v})\|_\otimes \,d\mu (y)
= \int_M \|(v_y-\bar{v})\|^2  d\mu (y) \leq 4D(t,M),
\end{equation}
showing that $\Sigma_\mu \in \vf_2(M) \otimes \vf_2(M)$. 

\begin{definition} \label{D:cvlinear}
For a probability measure $\mu \in \borel{M}$, abusing notation, define the {\em covariance operator} $\Sigma_\mu \colon \vf_2(M) \to \vf_2(M)$ by
\[
\Sigma_\mu (w) \coloneqq \int_M (v_y-\bar{v}) \innervf{v_y-\bar{v}}{w} \,d\mu(y).
\]
\end{definition}

$\Sigma_\mu \colon \vf_2(M) \to \vf_2(M)$ is a well-defined bounded operator because, by \eqref{E:bound},
\begin{equation} \label{E:est}
\|\Sigma_\mu w\|\leq \|w\| \int_M \|v_y-\bar{v}\|^2 d\mu (y) \leq 4 D(t,M) \|w\|.
\end{equation}

\begin{proposition} \label{P:tclass}
The covariance $\Sigma_\mu \colon \vf_2(M) \to \vf_2(M)$ is a positive semi-definite, self-adjoint, trace-class operator. 
\end{proposition}

\begin{proof}
For any $v \in \vf_2(M)$,  $\innervf{v}{\Sigma_\mu v} = \int_M \innervf{v_y-\bar{v}}{v}^2 d\mu(y) \geq 0$, showing that $\Sigma_\mu \geq 0$. For $v,w \in \vf_2(M)$, we have
\begin{equation} 
\innervf{w}{\Sigma_\mu v}  = \int_M \innervf{v_y-\bar{v}}{w} \innervf{v_y-\bar{v}}{v} d\mu(y) = \innervf{\Sigma_\mu w}{v},
\end{equation}
so that $\Sigma_\mu$ is self-adjoint. To verify the trace class property, let $\{v_i\}_{i=1}^\infty$ be an orthonormal basis of $\vf_2(M)$. As $\Sigma_\mu \geq 0$, it suffices to check that
\begin{equation}
\sum_{i=1}^\infty \innervf{v_i}{\Sigma_\mu v_i} < \infty.
\end{equation}
Since $\Sigma_\mu v_i = \int_M (v_y-\bar{v}) \innervf{v_y-\bar{v}}{v_i} d\mu (y)$, we have that
\begin{equation} \label{E:trace}
\begin{aligned}
\sum_{i=1}^\infty \innervf{v_i}{\Sigma_\mu v_i} &= \sum_{i=1}^\infty \int_M \innervf{v_y-\bar{v}}{v_i}^2 d\mu(y)
= \int_M \Big(\sum_{i=1}^\infty \innervf{v_y-\bar{v}}{v_i}^2 \Big)d\mu(y) \\
&= \int_M \|v_y-\bar{v}\|^2 d\mu(y) \leq 4 D(t,M),
\end{aligned}
\end{equation}
where we have used Parseval's identity and \eqref{E:bound}. This completes the proof.
\end{proof}

%------------------------

\section{PCA on Manifolds} \label{S:rpca}

Section \ref{S:cvt} shows that, for each $t>0$, $\Sigma_\mu \colon \vf_2(M) \to \vf_2(M)$ is a bounded, positive semi-definite, self-adjoint, trace-class operator. This implies that, enumerating the positive eigenvalues $\sigma_i^2$ in non-increasing order, we obtain a corresponding orthonormal eigenbasis $\{e_i\}$ of $\overline{\operatorname{Im}(\Sigma_\mu})$ that together with an orthonormal basis of $\ker \Sigma_\mu$ yield a complete orthonormal set for $\vf_2(M)$. Moreover, the trace of $\Sigma_\mu$ satisfies
\begin{equation} \label{E:trace2}
Tr(\Sigma_\mu) = \sum_i \innervf{e_i}{\Sigma_\mu e_i} =  \sum_i \sigma_i^2 < \infty.
\end{equation}
We refer to the eigenvector field $e_i$ as the {\em $i$th principal vector field} of $\mu$ (at scale $t>0$). If the multiplicity of the non-zero eigenvalue $\sigma_i^2$ is one (generically, this is the case), the principal vector field $e_i$ is uniquely defined up to sign. The $i$th {\em PC-score} of a point $y \in M$ is given by
\begin{equation}
\operatorname{PC}_i(y) \coloneqq \innervf{v_y-\bar{v}}{e_i}.   
\end{equation}

\begin{definition}
The {\em variance} of $\mu$, at scale $t>0$, is defined as $\sigma^2 \coloneqq Tr (\Sigma_\mu)$ and the {\em $i$th principal variance} as $\sigma_i^2$. 
\end{definition}

Note that from \eqref{E:trace} and \eqref{E:trace2} we can conclude that $\sigma^2 = Tr(\Sigma_\mu) \leq 4 D(t,M)$. This shows that, in this model, the variances of probability measures $\mu \in \borel{M}$ are uniformly bounded, and helps clarify why the covariance tensor is well defined for arbitrary $\mu$. Moreover, since the variance satisfies $\sigma^2 = \sum_{i=1}^\infty \sigma_i^2$, as in standard Euclidean PCA, if $\sigma^2 \ne 0$ we can interpret $\sigma_i^2/\sigma^2$ as the fraction of the total variance explained by the $i$th principal mode and $\sum_{i=1}^k \sigma_i^2/\sigma^2$ as the proportion of the total variance explained by the first $k$ principal components. 

%-----------------------

\section{The Gradient-Field Embedding} \label{S:gfembed}

In the terminology of data analysis, in Section \ref{S:rpca}, PCA for a probability measure $\mu \in \borel{M}$ was done by ``embedding'' $M$ into the ``feature space'' $\vf_2 (M)$, where the ambient Hilbert space structure allows us to perform standard linear PCA. Although this is a loose usage of the term embedding, in this section we show that the mapping $\Phi_t \colon M \to \vf_2 (M)$ given by
\begin{equation} \label{E:gfembedding}
\Phi_t (y) = v_y,
\end{equation}
where $v_y (x) = \nabla_x k_t (x,y)$, is an actual embedding. Here we prove this when $M$ is compact, deferring the more technical argument for non-compact manifolds of bounded geometry to the Appendix (see Theorem \ref{T:embed2}). From a practical perspective, through the embedding $\Phi_t$, we can approach computation of manifold PCA (discussed in the next section) as a kernel PCA problem. However, as already noted in the Introduction, a key difference is that, in the present formulation of manifold PCA, the embedding is known explicitly.

We begin with a couple of lemmas needed in the proof of the embedding theorem.

\begin{lemma} \label{L:injlemma}
Let $(M,g)$ be a complete Riemannian manifold and $y,y' \in M$. 
\begin{enumerate}[label={(\roman*)}]
\item If there exists $t_0 >0$ such that $k_{t_0} (x,y) = k_{t_0}(x,y')$, $\forall x \in M$, then $k_t (x,y) = k_t (x,y')$, $\forall x \in M$ and $\forall t>0$.
\item If there exists $t_0 >0$ such that $k_{t_0} (x,y) = k_{t_0}(x,y')$, $\forall x \in M$, then $y=y'$.
\end{enumerate}
\end{lemma}

\begin{proof}
(i) This follows from the semigroup property of the heat operator $e^{t \Delta}$. Indeed, for $0<t<t_0$,
\begin{equation}
e^{(t_0-t)\Delta} k_t (\cdot,y) = k_{t_0} (\cdot,y) = k_{t_0} (\cdot,y') = e^{(t_0-t)\Delta} k_t (\cdot,y').    
\end{equation}
Since $e^{(t_0-t)\Delta}$ is injective, it follows that $k_t (x,y) = k_t (x,y')$, $\forall x \in M$. For the case $t>t_0$, we have 
\begin{equation}
k_t (\cdot,y) = e^{(t-t_0)\Delta} k_{t_0} (\cdot,y) = e^{(t-t_0)\Delta} k_{t_0} (\cdot,y') = k_t (\cdot,y').
\end{equation}
%-------------
(ii) By (i), the heat kernel satisfies $k_t (x,y)=k_t(x,y')$, $\forall x\in M$ and $\forall t>0$. By symmetry, $k_t(y,x) = k_t(y',x)$. Therefore, for any continuous function with compact support $f \colon M \to \real$, we have
\begin{equation}
f(y)= \lim_{t\to 0^+} \int_M k_t(y,x) f(x) dvol(x) = \lim_{t\to 0^+} \int_M k_t(y',x) f(x) dvol(x) = f(y').
\end{equation}
This implies that $y=y'$ since continuous functions with compact support separate points in $M$.
\end{proof}

\begin{proposition} \label{P:injectivity}
If $(M,g)$ is a connected and complete Riemannian $d$-manifold of bounded geometry, $d \geq 1$, then the map $\Phi_t \colon M \to \vf_2(M)$ is injective for each fixed $t>0$.
\end{proposition}

\begin{proof}
Let $y,y' \in M$ be such that $\Phi_t (y) = \Phi_t (y')$; that is, $\nabla_x k_t (x,y) = \nabla_x k_t (x,y')$, $\forall x \in M$. This implies that there is a constant $c$ such that
\begin{equation} \label{E:difference}
k_t(x,y)-k_t(x,y') =c, 
\end{equation}
$\forall x \in M$. Since the heat kernel on manifolds of bounded geometry has the property that
\begin{equation}
\int_M k_t (x,y) dvol(x)=1,
\end{equation}
$\forall y \in M$, integrating \eqref{E:difference} with respect to $x$, we conclude that $c=0$. Therefore, $k_t(x,y) = k_t(x,y')$, $\forall x\in M$. Lemma \ref{L:injlemma} implies that $y=y'$, proving injectivity.    
\end{proof}

\begin{lemma} \label{L:span}
If $M$ is a connected and closed Riemannian manifold, then
\[
\mathrm{span}\,\{\nabla \phi_i(y) \colon i\geq1\}=T_y M,
\]
for any $y\in M$.
\end{lemma}

\begin{proof}
Suppose that the conclusion fails. Then there exist $y\in M$ and a nonzero vector $v\in T_yM$ such that
\begin{equation}
\inner{\nabla \phi_i(y)}{v}_y =0,
\end{equation}
for all $i \geq 1$. Equivalently, $d(\phi_i)_y (v)=0$, for all $i\geq 1$. Since $\phi_0$ is constant, we also have $d(\phi_0)_y(v)=0$.

Let $f\in C^\infty(M)$. The eigenfunctions $\phi_i$, $i \geq 0$, form a complete orthonormal set of $L^2(M)$ and elliptic regularity implies that their linear span is dense in $C^\infty(M)$ with respect to the $C^\infty$-topology. Hence, there exists a sequence
\begin{equation}
f_m=\sum_{i=0}^{N_m} a_i^m \phi_i,
\end{equation}
$m \geq 1$, such that $f_m \to f$ in $C^\infty(M)$.
For every $m\geq 1$,
\begin{equation}
d(f_m)_y(v)	=	\sum_{i=0}^{N_m} a_i^m d(\phi_i)_y (v) =0.
\end{equation}
Passing to the limit, we obtain $df_y(v)=0$ for every $f\in C^\infty(M)$. This is a contradiction since $v\neq 0$. Indeed, choosing local coordinates $(x^1,\ldots,x^d)$ around $y$, there exists an index $j$ such that $dx^j_y(v)\neq 0$. Let $\eta\in C_0^\infty(M)$ be a bump function satisfying $\eta\equiv1$ on a neighborhood of $y$, and set $f=\eta\,x^j$. Then $f\in C_0^\infty(M)$ and
\begin{equation}
df_y(v)=dx^j_y(v)\neq 0,
\end{equation}
a contradiction. 
\end{proof}

\begin{theorem} \label{T:embed1}
If $M$ is a connected, closed $d$-dimensional Riemannian manifold, $d \geq 1$, then the mapping $\Phi_t \colon M \to \vf_2 (M)$ is an embedding for each $t>0$.
\end{theorem}

\begin{proof}
Since $M$ is compact, it suffices to show that $\Phi_t$ is an injective immersion. Injectivity has been shown in Proposition \ref{P:injectivity} in the more general setting of manifolds of bounded geometry. To verify that $\Phi_t$ is an immersion, let $y\in M$ and $v\in T_y M$. We show that if $d(\Phi_t)_y (v)=0$, then $v=0$. From the spectral expansion of the heat kernel, we can write
\begin{equation} \label{E:vffield}
\nabla_x k_t(x,y) =\sum_{i=1}^{\infty}e^{-\lambda_i t} \phi_i (y) \nabla \phi_i(x).
\end{equation}
Differentiating $\Phi_t$ at $y$ in the direction $v$, we obtain
\begin{equation}
d(\Phi_t)_y(v)=\sum_{i=1}^{\infty}	e^{-\lambda_i t}
d(\phi_i)_y(v)	\nabla \phi_i .
\end{equation}
If $d(\Phi_t)_y (v)=0$, then
\begin{equation}
\sum_{i=1}^{\infty}	e^{-\lambda_i t} d(\phi_i)_y (v)\nabla \phi_i=0
\end{equation}
in $\vf_2 (M)$.	Taking the $\vf_2 (M)$-inner product with the gradient field $\nabla\phi_j$, we obtain
\begin{equation} \label{E:vanish1}
0 =\sum_{i=1}^{\infty}e^{-\lambda_i t}	d(\phi_i)_y(v)	\int_M
\inner{\nabla\phi_i (x)}{\nabla\phi_j(x)}_x\, dvol(x).
\end{equation}
Integration by parts yields
\begin{equation} \label{E:vanish2}
\begin{split}
\int_M \inner{\nabla\phi_i (x)}{\nabla \phi_j(x)}_x\, dvol(x)
&=- \int_M \phi_j(x)\Delta \phi_i (x)\, dvol(x) \\
&= \lambda_i \int_M \phi_i(x) \phi_j(x) \, dvol(x)
= \lambda_i\delta_{ij}.
\end{split}
\end{equation}
Equations \eqref{E:vanish1} and \eqref{E:vanish2} imply that $\lambda_j e^{-\lambda_j t}\, d(\phi_j)_y(v)=0$,
for any $j \geq 0$. Since the multiplicity of the zero eigenvalue of $\Delta$ coincides with the number of connected components of $M$, the hypothesis that $M$ is connected implies that $\lambda_j>0$ for $j\geq 1$. Thus, $d(\phi_j)_y(v)=0$, $\forall j \geq 1$. Equivalently, $\inner{\nabla\phi_j (y)}{v}_y=0$, $\forall j \geq 1$. By Lemma \ref{L:span}, $\{\nabla\phi_j(y) \colon j\geq 1\}$ spans $T_y M$, so $v=0$. This concludes the proof.
\end{proof}

%------------------

\section{Manifold PCA as Kernel PCA} \label{S:kpca}

Since Riemannian PCA at scale $t>0$ was done through the explicit embedding $\Phi_t \colon M \to \vf_2 (M)$, a possible pathway to computation is kernel PCA. Write
\begin{equation} \label{E:kernelt}
K_t (y,y') \coloneqq \innervf{\Phi_t (y)}{\Phi_t (y')} = \innervf{v_y}{v_{y'}}
\end{equation}
for the kernel associated with $\Phi_t$. Then, for each fixed $y' \in M$, using the fact that $\Delta e^{t\Delta} = e^{t\Delta}  \Delta$ and the semigroup property of $k_t$, we can express the kernel as

\begin{equation} \label{E:kpca}
\begin{split}
K_t (y,y') &= \int_M \inner{\nabla_x k_t(x,y)}{\nabla_x k_t(x,y')}_x \, dvol(x) \\
&= -\int_M k_t(x,y) \Delta_x k_t (x,y') \,dvol (x) \\
&= -\Delta_y \Big( \int_M k_t(y,x) k_t (x,y') \,dvol (x) \Big) \\
&= -\Delta_y  k_{2t} (y,y') = -\partial_s k_s (y,y')|_{s=2t} \\
&=- \frac{1}{2} \partial_t k_{2t} (y,y').
\end{split}
\end{equation}

\begin{example}
For $M = \real^d$, the kernel is
\begin{equation}
\begin{split}
K_t (y,y') &= -\frac{1}{2} \partial_t \left( \frac{1}{(8\pi t)^{d/2}} {e^{-\frac{\|y-y'\|^2}{8t}}}\right) \\
&= \frac{1}{4t(8\pi t)^{d/2}} \Big(d-\frac{\|y-y'\|^2}{4t}\Big) e^{-\frac{\|y-y'\|^2}{8t}}.
\end{split}
\end{equation}
\end{example}

If the heat kernel is available in closed form, \eqref{E:kpca} implies that we can carry out manifold PCA as standard kernel PCA, following the usual steps~\cite{Scholkopf1998}. A noteworthy difference is that since the embedding $\Phi_t$ is known, we can explicitly calculate the principal vector fields.  

%-----------------------

\section{Finite-Dimensional Reduction and RiePCA} \label{S:findim}

The formulation of covariance in Section \ref{S:cvt} accounts for covariation of a distribution $\mu$ about all points $x \in M$ because, in general, there may be no particularly salient reference points. In some situations, however, there may be a special set of points such as isolated local maximum points of the heat-kernel density estimator for $\mu$, at a given bandwidth, that can be interpreted as the modes of $\mu$ at that scale. For this reason, we introduce a variant of $\Sigma_\mu$ that measures covariation anchored at a fixed finite set $X_r = \{x_1, \ldots, x_r\} \subseteq M$, thus replacing a vector field over the entire manifold $M$ with a vector field over $X_r$. Such a vector field is given by a collection of tangent vectors $v_j \in T_{x_j} M$, $1 \leq j \leq r$. A virtue of this reduction is that it makes computation on high-dimensional manifolds possible, as densely sampling a vector field on $M$ is computationally intractable.  

For this $rd$-dimensional model, we can drop the assumptions that $M$ is complete and has bounded geometry. We also allow more general smooth potential functions, for example, proxies for the heat kernel that are more computable. Let $u \colon M \times M \to \real$ be a smooth, non-negative potential function such that, for each $x \in M$, the inequality
\begin{equation} \label{E:gradu}
\|\nabla_x u(x,y)\|_x \leq C_x    
\end{equation}
holds for all $y \in M$, where $C_x>0$ is a constant that may depend on $x$. By \eqref{E:gradhk}, this is satisfied by the heat kernel. Let $T(M,r)$ be the vector space $T(M,r) \coloneqq T_{x_1} (M) \times \ldots \times T_{x_r} (M)$ equipped with the inner product
\begin{equation}
\innervf{v}{w}_r \coloneqq \inner{v_1}{w_1}_{x_1} +\ldots +\inner{v_r}{w_r}_{x_r}.
\end{equation}
For each fixed $y \in M$, denote by $v_y^r$ the discrete vector field
\begin{equation} \label{E:dvf}
v_y^r = (\nabla_x u(x_1,y), \ldots ,\nabla_x u(x_r,y)) \in T(M,r)
\end{equation}
supported on the anchor set $X_r$. The {\em covariance tensor} $\Sigma_{\mu,r} \in T(M,r) \otimes T(M,r)$, with respect to the potential $u$, is defined as
\begin{equation} \label{E:fcovt}
\Sigma_{\mu,r} \coloneqq \int_M (v_y^r - \bar{v}) \otimes (v_y^r - \bar{v}) d\mu(y),
\end{equation}
where $\bar{v}$ is the mean discrete vector field over $X_r$. The tensor $\Sigma_{\mu,r}$ is well defined by \eqref{E:gradu}. The corresponding self-adjoint, positive semi-definite linear mapping $\Sigma_{\mu,r} \colon T(M,r) \to T(M,r)$ is given by
\begin{equation} \label{E:fcovop}
\Sigma_{\mu,r} (v) \coloneqq \int_M (v_y^r - \bar{v}) \innervf{v_y^r - \bar{v}}{v}_r \,d\mu (y).
\end{equation}
Of particular practical interest is the covariance $\Sigma_{n,r}$ of an empirical measure $\mu_n = \sum_{i=1}^n \delta_{y_i}/n$, where $y_1, \ldots, y_n \in M$. In this case, \eqref{E:fcovt} and \eqref{E:fcovop} can be written as
\begin{equation}
\Sigma_{n,r} = \frac{1}{n} \sum_{i=1}^n (v_{y_i}^r- \bar{v}) \otimes (v_{y_i}^r - \bar{v})
\mbox{\quad and \quad}
\Sigma_{n,r}  (v) = \frac{1}{n} \sum_{i=1}^n (v_{y_i}^r - \bar{v}) \innervf{v_{y_i}^r- \bar{v}}{v}_r,
\end{equation}
respectively.

As before, the principal components $e_i \in T(M,r)$ are the unit-length eigenvectors of $\Sigma_{\mu,r}$ associated with the positive eigenvalues $\sigma_{i,r}^2 > 0$ and can be completed to an orthonormal basis of eigenvectors of $\Sigma_{\mu,r}$ by adding an orthonormal set that spans the nullspace of $\Sigma_{\mu,r}$. The variance of $\mu$ is defined as $\sigma_r^2 \coloneqq Tr(\Sigma_{\mu,r})$ and satisfies $\sigma_r^2 = \sum_{i=1}^{rd} \sigma_{i,r}^2$. We refer to this version of principal component analysis as RiePCA.

\medskip

Next, we describe the algorithmic steps for RiePCA anchored at $X_r = \{x_1, \ldots, x_r\}$, applied to a probability measure $\mu$ supported on a finite set $\{y_1, \ldots, y_n\} \subseteq M$. Write the probability distribution as $\mu = \sum_{j=1}^n \mu_j \delta_{y_j}$ and let $D_\mu$ be the diagonal matrix $D_\mu = \operatorname{diag}(\mu_1, \ldots, \mu_n)$.

\begin{algorithm} \label{A:riepca}
(RiePCA Algorithm)

\begin{enumerate}%[label={(\roman*)}]
    \item Calculate the vector fields $v_j \coloneqq v^r_{y_j}= (\nabla_x u(x_1,y_j), \ldots, \nabla_x u(x_r,y_j))$, $1 \leq j \leq n$.
    \item Center the vector fields: $\tilde{v}_j = v_j - \bar{v}=
    v_j - \sum_{j=1}^n v_j\mu_j$.
    \item Calculate the Gram matrix $G$ whose $(i,j)$-entry is $g_{ij} = \innervf{\tilde{v}_i}{\tilde{v}_j}_r$.
    \item Set $K \coloneqq D_\mu^{1/2}G D_\mu^{1/2}$.
    \item Diagonalize $K$ and index the eigenvalues $\sigma_i^2\coloneqq \sigma_{i,r}^2$ in non-increasing order: (i) the non-zero eigenvalues of $K$ are the same as the non-zero eigenvalues of the covariance operator; (ii) if $f_i = (f^i_1, \ldots, f^i_n)$ is an eigenvector of $K$ with eigenvalue $\sigma^2_i \ne 0$, then $\hat{e}_i = \sum_{j=1}^n \sqrt{\mu_j} f^i_j \tilde{v}_j$ is a corresponding eigenvector of the covariance, so that $e_i = \hat{e}_i/\|\hat{e}_i\|$ is a principal vector field of unit norm.
    \item The PC-score of a point $y \in M$ along the $i$th principal direction is given by
    \[
    \operatorname{PC}_i (y)=\innervf{v_y-\bar{v}}{e_i}_r = \innervf{(\nabla_x u(x_1,y), \ldots, \nabla_x u(x_r,y))-\bar{v}}{e_i}_r.
    \] 
\end{enumerate}
\end{algorithm}

\begin{remark}
Algorithm \ref{A:riepca} approaches RiePCA through the Gram matrix. However, if we have an orthonormal basis for $T(M,r)$, for example, one gotten from orthonormal bases for the tangent spaces $T_{x_i}M$, $1 \leq i \leq r$, an alternative is to represent the vector fields $v_j$ in coordinates through projections onto the basis elements and compute the covariance matrix directly. This is how we approach RiePCA for data in Euclidean space.
\end{remark}

\begin{example} \label{EX:m1}
We apply RiePCA with the heat kernel as a potential function to a synthetic dataset consisting of two disjoint elliptical regions in $\mathbb{R}^2$, one thinner and more elongated than the other. Each ellipse is uniformly sampled with the same number of points ($700$ data points per region) and the probability measure $\mu$ attributes the same weight to all points, so that each elliptical cluster contains $50\%$ of the probability mass. At $t=1$, the heat kernel density estimator for $\mu$ yields one local maximum in each ellipse, which are used as anchor points. This leads to a $4$-dimensional representation of the data (2D vector fields over two anchor points). Figure~\ref{fig:ellipses}a shows each of the first four principal vector fields (PC1–PC4) as a single arrow at one of the reference points because all four principal vector fields have relatively negligible magnitude at the other point. At this scale, the principal directions adapt to the local geometry and capture variation within each cluster. PC1 detects the major axis of the longer ellipse, PC2 and PC3 detect the major and minor axes of the rounder cluster, and PC4 is about variation orthogonal to PC1 in the thinner ellipse. As $t$ increases, the heat diffusion progressively suppresses the local structure of the dataset. At $t=12$, the density estimator has a unique maximum, and the first two principal components capture more of the global structure of the dataset, as depicted in Figure~\ref{fig:ellipses}b, similar to the behavior of standard PCA with the single reference point falling near the mean of the data. Figure~\ref{fig:ellipses}c provides heat maps of the corresponding PC-scores for all data points at $t=1$.
%----------------
\begin{figure}[ht!]
\begin{tabular}{ccc}
\begin{tabular}{c}
\includegraphics[width=0.24\textwidth]{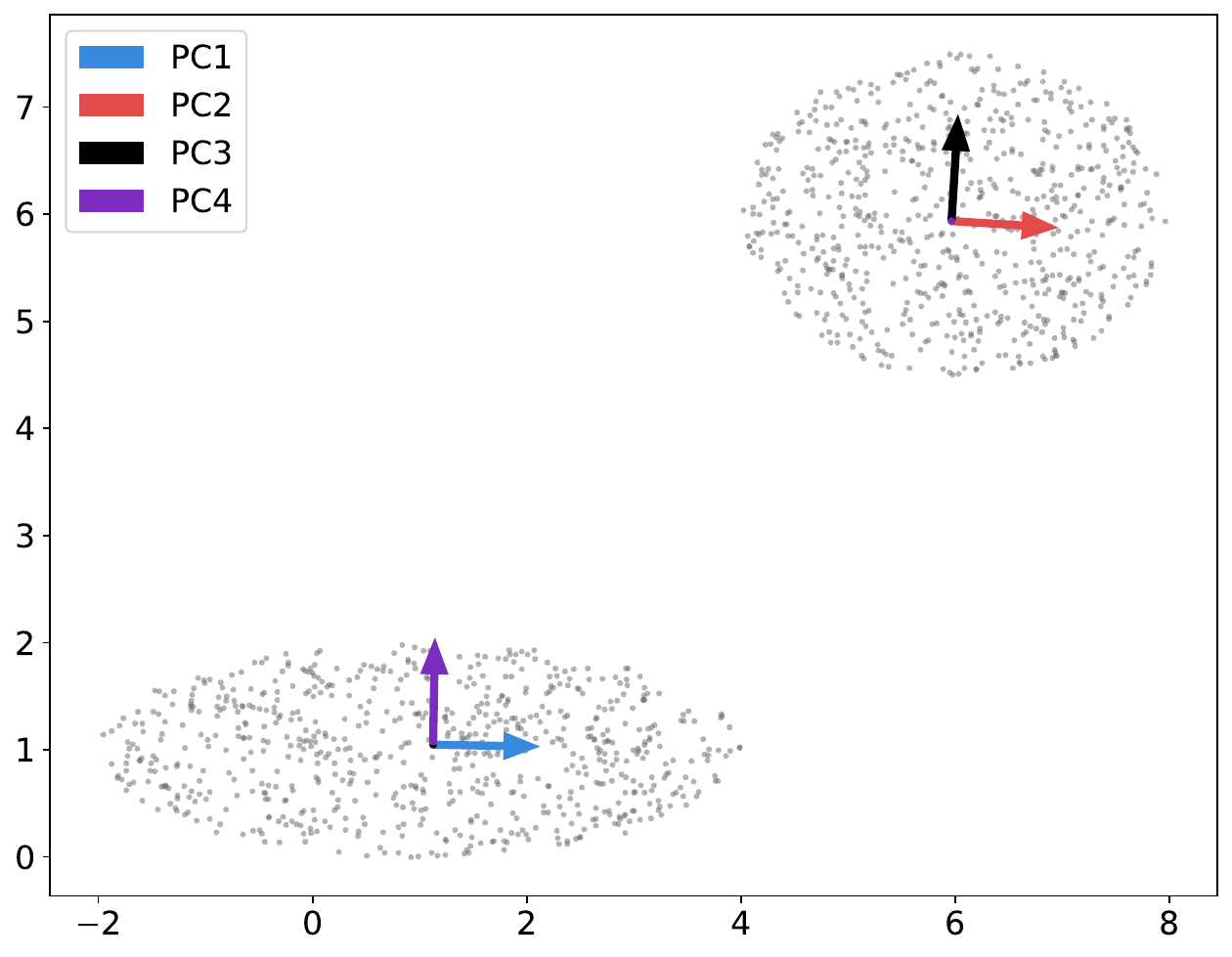} 
\end{tabular}
&
\begin{tabular}{c}
\includegraphics[width=0.24\textwidth]{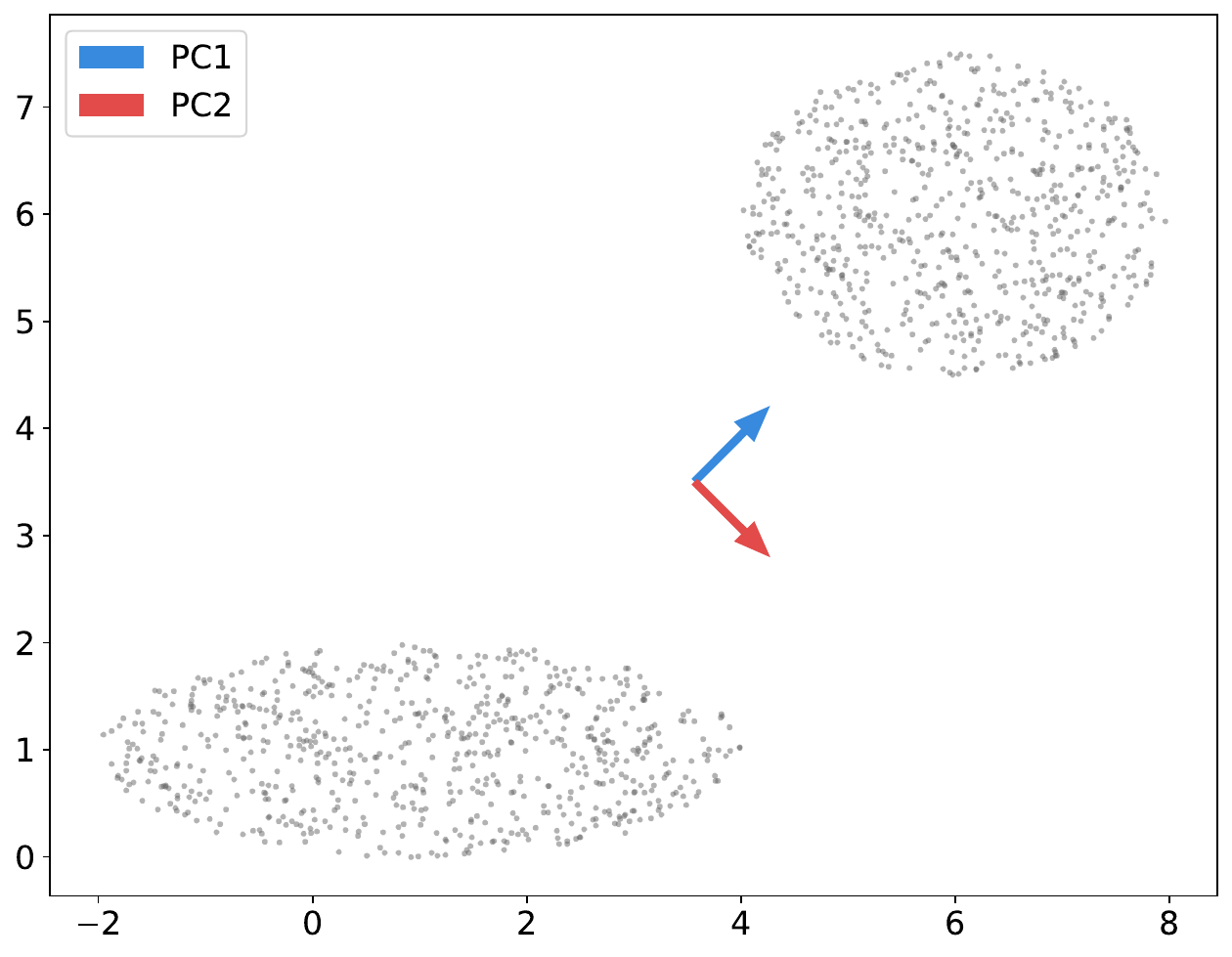} 
\end{tabular}
&
\begin{tabular}{c}
\includegraphics[width=0.36\textwidth]{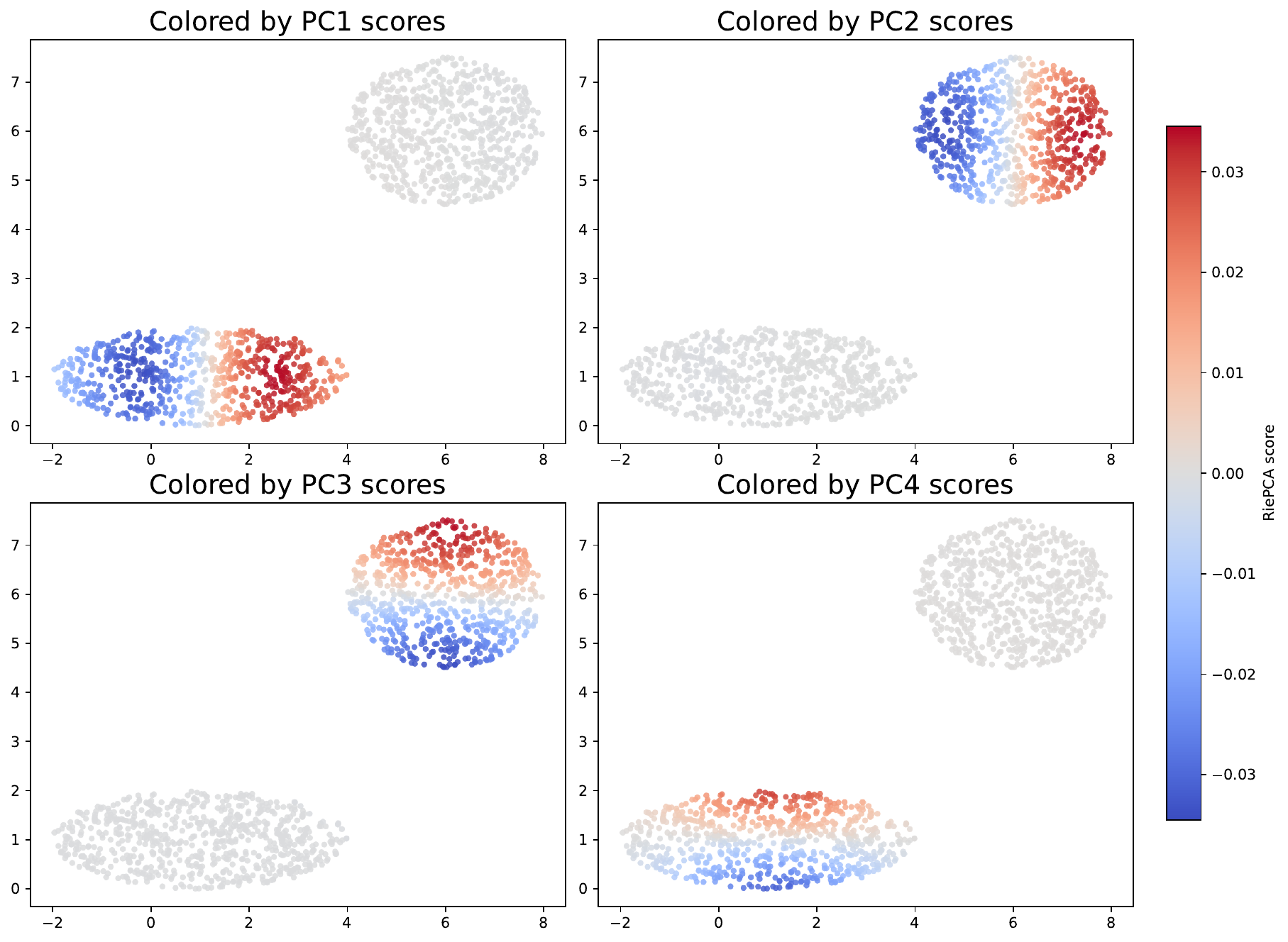}
\end{tabular} \\
(a) & (b) & (c)
\end{tabular}
\caption{(a) The principal vector fields PC1-PC4 at $t=1$; (b) the principal vector fields PC1 and PC2 at $t=12$; (c) the dataset colored by their first four PC-scores at $t=1$.}
\label{fig:ellipses}
\end{figure}

\end{example}

\begin{example} \label{EX:m2}
This is a non-Euclidean analog of Example \ref{EX:m1}. The synthetic dataset comprises two well-separated elliptical regions ($350$ data points per region) on the unit sphere $S^2$, where each region is generated by sampling an anisotropic Gaussian distribution on the tangent spaces at the points $p=(0,0,1)$ and $p'=(0.4/\sqrt{0.52},0,-0.6/\sqrt{0.52})$ and mapping the samples onto the sphere via the exponential map. The probability distribution assigns $30\%$ and $70\%$ of the total mass to the blue and red clusters, respectively, with the mass distributed uniformly among the points within each cluster. The potential function used is $u (x,y) = e^{-\frac{(\arccos(x\cdot y))^2}{4t}}/(4\pi t)$, $t=0.15$. The two reference points selected are the local maxima of the corresponding kernel density estimator at scale $t$, one within each cluster. The potential function is not smooth for pairs $(x,y)$ on the cut locus but this does not affect this example as all relevant pairs are away from the singular pairs. The principal vector fields PC1--PC4 are depicted as single arrows in Figure~\ref{fig:2:d}, as they have negligible magnitude at one of the points. The principal vector fields show that, at appropriate scales, RiePCA captures the local geometry of the clusters, placing the cluster with larger mass higher in the hierarchy. The colors show the assignment of probability mass to each data point.
%-------------------
\begin{figure*}[htbp!]
    \centering
    \includegraphics[width=\textwidth]{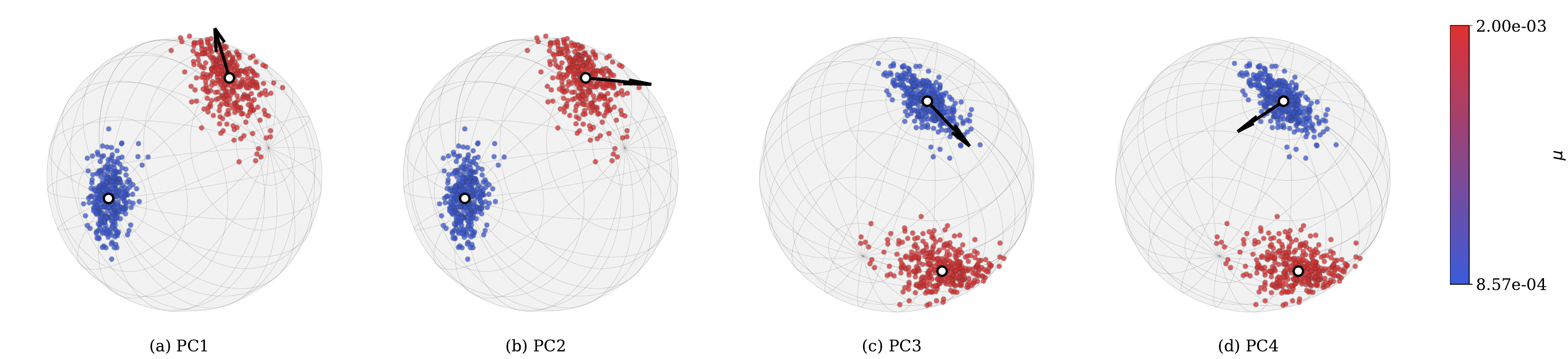}
    \caption{The colors indicate the mass assigned to each data point and the arrows show the first four RiePCA principal vector fields at $t=0.15$.}
    \label{fig:2:d}
\end{figure*}
%---------------
\end{example}

\begin{example}
The primary goal of this example is to illustrate how RiePCA can be used as a manifold learning technique for uncovering nonlinear patterns in data. Here, the ambient manifold is $\real^3$ and the dataset consists of $2000$ points sampled uniformly (with respect to the angular coordinates) from a torus in $\real^3$ with major radius $R=2$ and minor radius $r=0.8$, with additive white Gaussian noise at level $0.05$; see Figure~\ref{fig:3a}. In this analysis, the probability measure attributes the same mass to all points. We select 25 reference points, sampling from a small neighborhood of local maxima of the heat-kernel density estimator restricted to the data points at $2t=0.02$ (hence $t=0.01$). Thus, the space of vector fields has dimension $25 \times 3 = 75$, yielding up to 75 principal directions. This is in sharp contrast with classical PCA that produces at most three principal components because the data points lie in $\real^3$. 

To facilitate visualization of the behavior of the principal components, Figure~\ref{fig:3a} shows the data colored by the longitudinal angle $\theta$ and meridional angle $\varphi$, respectively. Figure~\ref{fig:3b} shows RiePCA projections, at scale $t=2.75$, onto: (i) the PC1--PC2 plane that captures the global structure around the longitudinal circles; (ii) the PC3--PC4 plane that reveals the meridional circular structure.  Figure~\ref{fig:4} shows the PC scores for the first four principal components.
%------------------
\begin{figure*}[htbp!]
    \centering
    \begin{subfigure}[t]{0.48\textwidth}
        \centering
        \includegraphics[width=\textwidth]{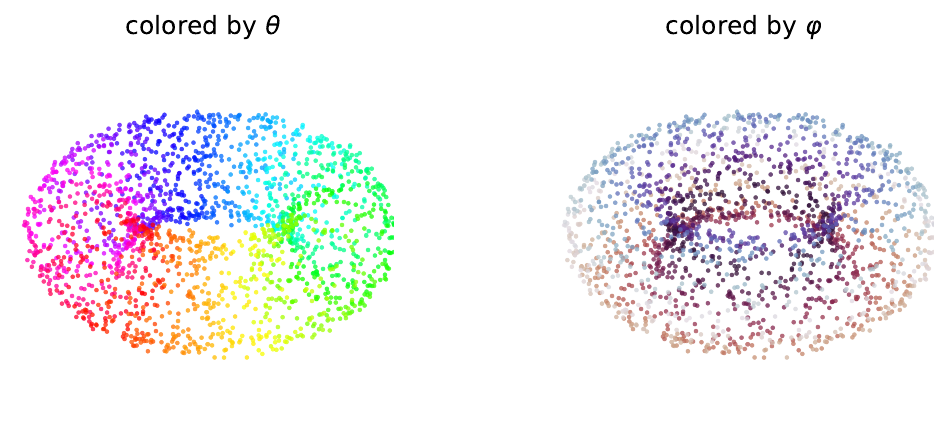}
        \caption{}
        \label{fig:3a}
    \end{subfigure}%
    \hfill
    \begin{subfigure}[t]{0.48\textwidth}
    \centering
    \includegraphics[width=\textwidth]{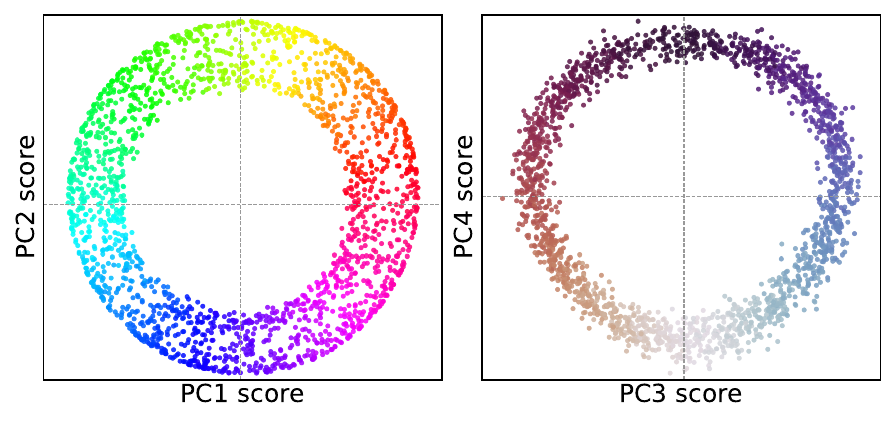}
    \caption{}
    \label{fig:3b}
    \end{subfigure}
\caption{(a) The torus dataset; (b) projection onto the first four principal directions at $t=2.75$.}
\end{figure*}
        \begin{figure*}[htbp!]
        \centering
        \includegraphics[width=0.8\textwidth]{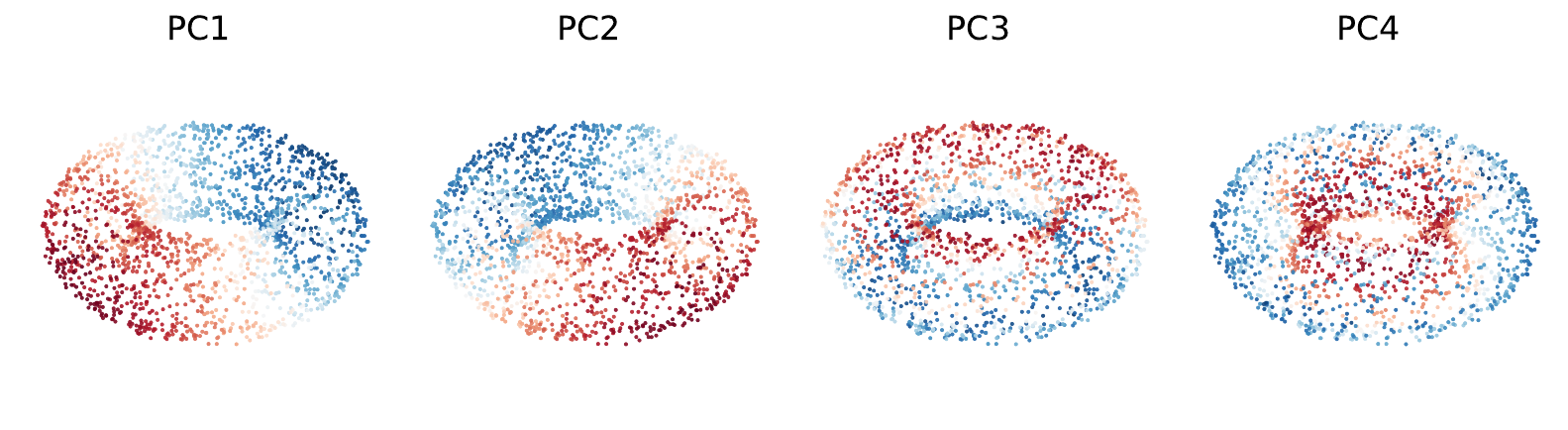}
        \caption{Heat map of PC-scores at $t=2.75$.}
        \label{fig:4}
\end{figure*}
  
It is instructive to compare our method with kernel PCA using the Euclidean heat kernel with various bandwidths $\sigma$. We observed that for the range $\sigma\in [0.7,2]$, the meridional circular structure first appears in the PC$5$-PC$6$ projection, whereas RiePCA detects the same structure earlier in the hierarchy, at PC3 and PC4. For larger bandwidths (such as $\sigma=4$), the meridional circle is not recovered in any of the consecutive pairwise PC projections from PC1–PC2 through PC5–PC6, likely because at those scales the circular structure becomes too small to be detected. These observations suggest that, for this dataset, RiePCA is more effective in learning the key geometric and topological features of the data.
\end{example}

%--------------------

\section{Covariance and PCA on Graphs} \label{S:graphpca}

Let $G=(V,E,w)$ be a connected, weighted finite simple graph, where $V$ denotes the vertex set, $E$ denotes the edge set, and $w \colon E\to\real_{>0}$ assigns a positive real number to each edge that represents its geometric length. In the Riemannian formulation, the covariance of a probability measure on a manifold $M$ was defined by representing points in $M$ as vector fields induced by the heat kernel $k_t$. For a graph $G$, we define covariance with respect to more general non-negative potential functions $u \colon V \times V \to \real$, with the heat kernel remaining a potential of central interest. 

We fix a labeling $V=\{x_1, \ldots, x_n\}$ of the vertices of $G$ and a labeling $E=\{e_1, \ldots, e_m\}$ of the edges of $G$. We also fix a reference orientation for each edge of the graph and adopt the notation $\orient{E}$ for the oriented edge set with the same indexing as $E$. Under these assumptions, we:
\begin{enumerate}[label={(\roman*)}]
\item identify a function $h \colon V \to \real$ with a vector $h \in \real^n$ whose $i$th coordinate is $h_i = h(x_i)$;
\item write a probability measure $\mu$ on $V$ as a vector $\mu \in \real^n$, whose $i$th entry is $\mu_i = \mu(x_i)$, satisfying $\mu_i \geq 0$ and $\sum_{i=1}^n \mu_i =1$;
\item represent a potential $u \colon V \times V \to \real$ as an $n \times n$ matrix $U = (u_{ij})$ whose $(i,j)$-entry is $u_{ij} = u(x_i,x_j)$;
\item encode the edge weights in an $m \times m$ diagonal matrix $W$ whose diagonal entries are $w_{\ell\ell} = w (e_\ell)$, $1 \leq \ell \leq m$;
\item associate with $G$ a weighted $n \times m$ incidence matrix $K=(k_{i\ell})$ whose rows are indexed by the vertices of $G$ and whose columns are indexed by the edges of $G$. Letting $w_\ell \coloneqq w(e_\ell)$, the $(i,\ell)$-entry of $K$ is
\begin{equation} \label{E:incidence}
k_{i\ell} \coloneqq
\begin{cases}
\frac{1}{w_\ell}, & \text{if $x_i$ is the target vertex of $e_\ell$;}\\
\frac{\!\!-1}{w_\ell}, & \text{if $x_i$ is the source vertex of $e_\ell$;}\\
\ 0, & \text{otherwise}.
\end{cases}
\end{equation}
\end{enumerate}

A vector field $v$ on $G$ is modeled as an assignment of an orientation and a magnitude for each edge. Once a reference orientation $\orient{E}$ has been fixed, a {\em vector field} can be represented by a map $v\colon \orient{E} \to \real$. For $e\in \orient{E}$, $v(e)>0$ indicates that the vector field has magnitude $v(e)$ and points in the same direction as the (fixed) orientation of $e$, whereas $v(e)< 0$ indicates an orientation opposite to that of $e$ and magnitude $|v(e)|$. The fixed orientations for the edges let us write a vector field $v$ as an $m$-vector whose $\ell$th entry is $v_\ell = v(e_\ell)$, $1 \leq \ell \leq m$.

We denote by $\vf_G$ the vector space of all vector fields on $G$ and equip it with the inner product
\begin{equation}
\innervf{v}{v'}
\coloneqq \sum_{\ell=1}^{m} w_\ell \,v_\ell v'_\ell,
\end{equation}
for all $v,v' \in \vf_G$, an expression that is clearly independent of the base orientation $\orient{E}$. We identify the inner product space $\big(\vf_G, \innervf{\,}{}\big)$ with $\real^m$ (with the usual dot product) under the linear mapping $\psi \colon \vf_G \to \real^m$ given by
\begin{equation} \label{E:iso}
\psi(v) \coloneqq W^{1/2} v = 
\begin{pmatrix}
\sqrt{w_1} v_1 \\
\vdots \\
\sqrt{w_m} v_m
\end{pmatrix}.
\end{equation}
It is simple to verify that $\psi$ is an isometry. Indeed, $\psi$ is clearly surjective and
\begin{equation}
\psi(v) \cdot \psi(v') = \sum_{\ell=1}^m (\sqrt{w_\ell} v_\ell)(\sqrt{w_\ell} v'_\ell) = \sum_{\ell=1}^m w_\ell v_\ell v'_\ell = \innervf{v}{v'},    
\end{equation}
for all $v, v' \in \vf_G$.

To formulate a graph version of covariance, we employ the discrete gradient operator $\nabla$ that maps a function $h\colon V \to \real$ to the vector field $\nabla h \in \vf_G$ given on the oriented edge $e_\ell=(x_i,x_j) \in \orient{E}$ by
\begin{equation}
\nabla h(e_\ell)\coloneqq \frac{1}{w_\ell}(h_j-h_i).
\end{equation}
In coordinates, the discrete gradient operator $\nabla \colon \real^n \to \real^m$ is represented by the matrix $K^\top$; that is, $\nabla h= K^\top h$.

We now introduce the discrete analog (relative to the potential $u \colon V \times V \to \real$) of the vector fields $v_y$ of Definition \ref{D:cvt}. For each $x_j\in V$, let $u^j \colon V \to \real$ be the function given in vector notation by $u^j_i = u_{ij}=u(x_i,x_j)$. Define the vector field $v^j \in \vf_G$ by
\begin{equation}
v^j= \nabla u^j = K^\top u^j.
\end{equation}
Under the isomorphism defined in \eqref{E:iso}, we identify $v^j$ with the vector $\psi(v^j) = W^{1/2}K^\top u^j \in \real^m$.

%-----------------
\begin{remark} \label{R:laplacian}
In the manifold case, as in \eqref{E:kernelt} and \eqref{E:kpca}, integration by parts gives
\begin{equation}
\innervf{v_y}{v_{y'}}= -\int_M k_t(x,y) \Delta k_t (x,y') \,dvol (x),
\end{equation}
whose graph analog is
\begin{equation}
\psi(v^j) \cdot \psi(v^{j'})=\innervf{v^j}{v^{j'}} = ({u^j})^\top KWK^\top {u^{j'}}.    
\end{equation}
Thus, with the usual sign convention, the discrete Laplacian is given by $L=KWK^\top$ with the associated heat operator $e^{-tL}$. It can be easily verified that $L$ is the standard graph Laplacian with edge weights $1/w_{\ell}$.
\end{remark}
 
\begin{definition}
Let $G=(V,E,w)$ be a weighted graph and $\mu$ be a probability measure on $V$. The {\em covariance matrix} $\Sigma_{\mu}$, relative to the potential $u \colon V \times V \to \real$, is defined as
\begin{equation*}
\begin{split}
    \Sigma_{\mu} &\coloneqq
    \sum_{j=1}^n W^{1/2}(v^j-\bar{v}) (v^j-\bar{v})^{\top} W^{1/2}\mu_j= W^{1/2} \Big(\sum_{j=1}^n (v^j-\bar{v})(v^j-\bar{v})^{\top} \mu_j\Big) W^{1/2},
\end{split}
\end{equation*}
where $\bar{v}=\sum_{j=1}^n v^j\mu_j$ is the {\em mean vector field}.
\end{definition} 

Having the covariance matrix defined, PCA can be done by diagonalizing $\Sigma_\mu$. The orthonormal eigenvectors of $\Sigma_\mu$ associated with the non-zero eigenvalues $\sigma^2_i$, under the isomorphism $\psi$,  yield principal vector fields $e_i$ that determine the principal directions of variation of $\mu$. The PC score of a vertex $x_j \in V$ along the $i$th principal direction is given by $\innervf{v^j-\bar{v}}{e_i}$. As usual, we label the eigenvalues in non-increasing order, and $\sigma_i^2$ gives the variance of the data projected onto the $i$th principal direction. 

\begin{remark} 
Relabeling the vertices of $G$ has no effect on the covariance matrix, whereas relabeling the edges induces a permutation of the coordinates of the eigenvectors of $\Sigma_\mu$. However, as expected, the principal vector fields remain unchanged. As for changes in edge orientation, let $\Sigma_\mu$ and $\Sigma'_\mu$ be the covariance matrices of $\mu$ relative to two different choices of orientations for the edges of $G$. Define $S$ to be the diagonal $m \times m$ matrix whose diagonal entries are $s_{\ell\ell} = 1$ if the orientations of the edge $e_\ell$ agree and $s_{\ell\ell} = -1$, otherwise. Clearly, $S^2 = I_m$. Then, the corresponding incidence matrices $K$ and $K'$, defined in \eqref{E:incidence}, satisfy $K'=KS$. Using this, one can verify that $\Sigma'_\mu = S\Sigma_\mu S=S^{-1} \Sigma_\mu S$, which implies that the eigenvalues of $\Sigma_{\mu}$ and $\Sigma'_\mu$ are the same and $v$ is an eigenvector of $\Sigma_\mu$ if and only if $Sv$ is an eigenvector of $\Sigma'_\mu$. Thus, only the representation in Euclidean coordinates of the principal vector fields changes, not the vector fields themselves.
\end{remark}

Algorithmically, one could approach GraphPCA through the Gram matrix, as in Algorithm \ref{A:riepca} for RiePCA. However, since the covariance matrix can be calculated explicitly for GraphPCA, we have the option of diagonalizing the covariance matrix, as described in the next algorithm.

\begin{algorithm} \label{A:graphpca}
(GraphPCA Algorithm) 

\begin{enumerate}%[label={(\roman*)}]
    \item Calculate the vector fields $v^j= v^{x_j}= K^\top u(\cdot,x_j)$, $1 \leq j \leq n$.
    \item Center the vector fields by setting $\tilde{v}^j = v^j - \bar{v}=
    v^j - \sum_{j=1}^n v^j\mu_j$.
    \item Compute the eigenvalues and eigenvectors of $\Sigma_\mu$ by applying SVD to the matrix $A=W^{1/2}\Lambda D_{\mu}^{1/2}$, where $\Lambda$ is a $m\times n$ matrix whose $(i,j)$-entry is $\lambda_{ij}=\tilde{v}^j_i$ and $D_\mu$ is the diagonal matrix $D_\mu = \operatorname{diag}(\mu_1, \ldots, \mu_n)$. The $i$th eigenvalue of $\Sigma_{\mu}$ is the square of the $i$th singular value of $A$ and the $i$th eigenvector of $\Sigma_{\mu}$ is the $i$th left singular vector of $A$.
    \item For $1\leq i\leq \operatorname{rank} A$, the $i$th PC-score of a vertex $x_j$ is $\operatorname{PC}_i (x_j) =(e^i)^\top W^{1/2}\tilde{v}^j$, where $e^i$ is the $i$th eigenvector of $\Sigma_{\mu}$.
\end{enumerate}
\end{algorithm}
%-----------------
\begin{example} \label{EX:circle}
This example illustrates GraphPCA applied to a discrete representation of the uniform distribution (normalized arc-length measure) on the unit circle. The underlying graph is a cycle with 200 vertices placed on the circle at angles $\theta_j = \pi j/100$, $0 \leq j \leq 199$, with neighboring vertices connected by edges of uniform weight $w=\pi/100$, the geodesic distance between neighboring vertices. The distribution $\mu$ attributes each vertex a probability mass $\mu_j=1/200$. We construct the graph Laplacian, as described in Remark \ref{R:laplacian}, and take the associated heat kernel $k_t$. 
%--------------
\begin{figure*}[htbp!]
\centering
\begin{subfigure}[t]{\textwidth}
\centering
\includegraphics[width=\textwidth]{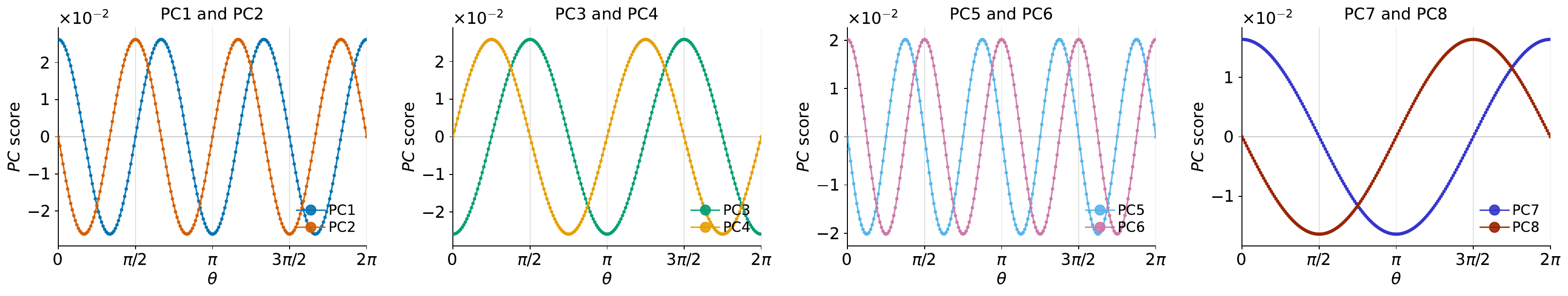}
\caption{$t=2.5$}
\label{fig:oscillations_t_2p5}
\end{subfigure}
\begin{subfigure}[t]{0.5\textwidth}
\centering
\includegraphics[width=\textwidth]{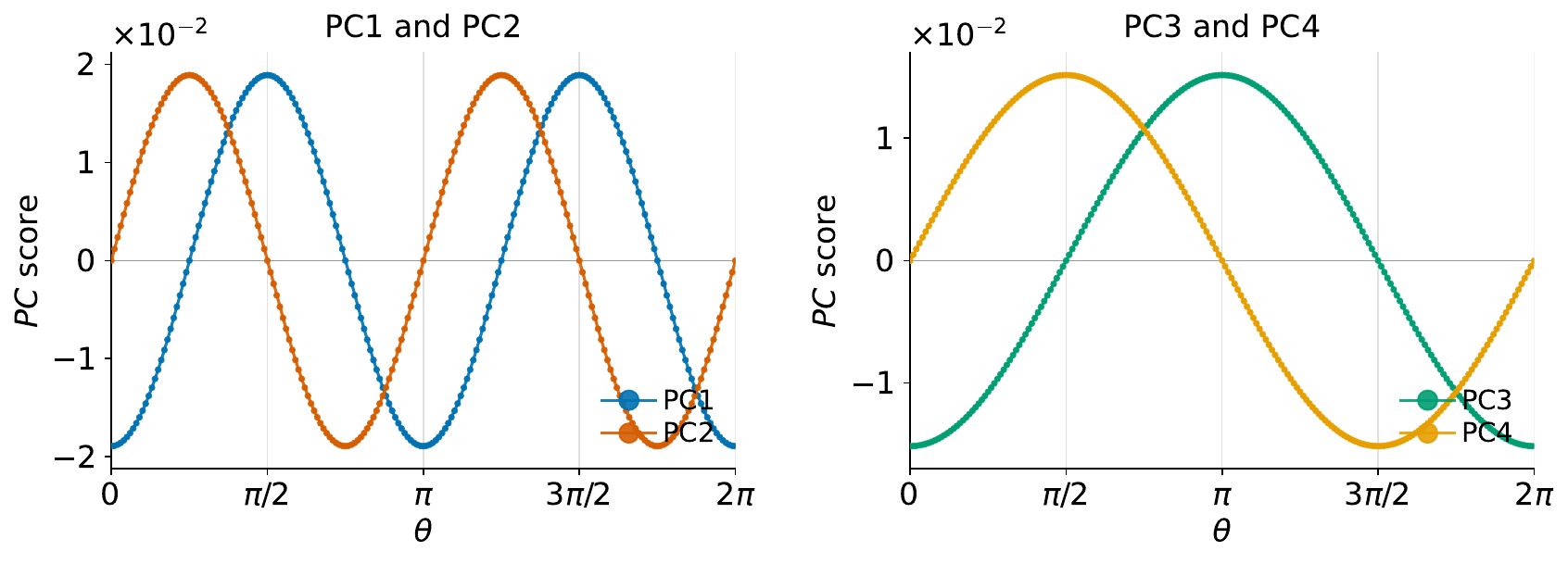}
\caption{$t=5$}
\label{fig:oscillations_t_5}
\end{subfigure}%
\hfill
\begin{subfigure}[t]{0.5\textwidth}
\centering
\includegraphics[width=\textwidth]{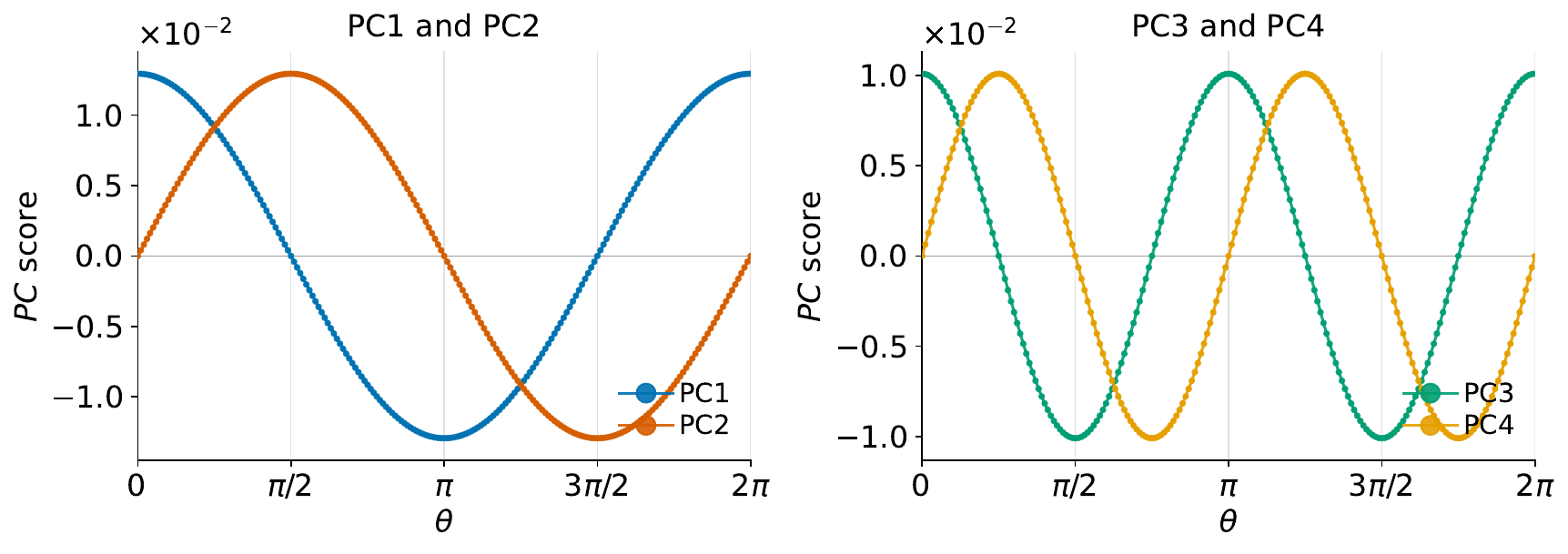}
\caption{$t=10$}
\label{fig:oscillations_t_10}
\end{subfigure}
\caption{PC scores for a discrete representation of the uniform distribution on the unit circle: (a) the first eight principal components at $t=2.5$; (b) and (c) the first four principal components at $t=5$ and $t=10$, respectively.}
\label{fig:pc_scores}
\end{figure*}
%------------------
Figure~\ref{fig:pc_scores} plots the PC-scores of the vertices of the graph at different scales, showing the first eight principal components at $t=2.5$ in panel (a) and the first four principal components at $t=5$ in panel (b) and at $t=10$ in panel (c), as a function of the angle $\theta_j$. The graph is drawn as a continuous curve to facilitate visualization. We observe a similarity between the PC-scores and the eigenfunctions of the Laplacian, except that at some scales high-frequency eigenfunctions may appear before lower-frequency ones.
\end{example}

The behavior of the PC-scores observed in Example \ref{EX:circle} is no coincidence. For the normalized volume measure of a closed manifold, the principal directions can be chosen so that their PC-scores are proportional to the eigenfunctions of the Laplace-Beltrami operator. However, at each scale $t>0$, the principal variance hierarchy may reorder some of the eigenfunctions. To verify these facts, for a closed connected manifold $M$, let $\{\phi_i \colon i \geq 0\}$ be a complete orthonormal set of eigenfunctions of $\Delta$ such that $\Delta \phi_i = -\lambda_i \phi_i$, with $\lambda_i \leq \lambda_{i+1}$.

\medskip
Denote by $H$ the subspace $H \coloneqq \overline{\operatorname{span}\{\nabla \phi_i \colon i \geq 1\}} \subseteq \vf_2 (M)$.

\begin{proposition}
Let $M$ be a closed, connected $d$-dimensional Riemannian manifold, $d \geq 1$, and $\mu = vol/V_M$, where $V_M$ is the volume of $M$. For $t>0$, the gradient fields $e_i =\nabla \phi_i/\sqrt{\lambda_i}$, $i \geq 1$, form an orthonormal basis of $H$ consisting of principal vector fields of $\Sigma_\mu$. The corresponding PC-score functions are
\[
\operatorname{PC}_i (y) = \sqrt{\lambda_i} e^{-\lambda_i t} \phi_i(y)
\]
and the principal variances are $\sigma_i^2 = \lambda_i e^{-2t \lambda_i}/V_M$. Thus, for $t\geq 1/(2\lambda_1) >0$, the principal variances are ordered as $\lambda_i$, but higher-frequency eigenfunctions may appear earlier in the ordering of the principal directions for smaller values of $t$.
\end{proposition}

\begin{proof}
As in \eqref{E:vffield}, write the vector fields $v_y$ as
\begin{equation}
v_y=\nabla_x k_t (\cdot,y) =
\sum_{i=1}^\infty e^{-\lambda_i t} \phi_i(y) \nabla \phi_i.
\end{equation}
Then,
\begin{equation}
\int_M v_y d\mu(y) = 
\frac{1}{V_M} \sum_{i=1}^\infty e^{-\lambda_i t} \nabla \phi_i (x) \int_M \phi_i (y) dvol(y)=0 
\end{equation}
since $\phi_i$ is orthogonal to constant functions for $i\geq 1$. Thus, the fields $v_y$ are $\mu$-centered. It follows from the definition of $\Sigma_\mu \colon \vf_2(M) \to \vf_2 (M)$ that $\operatorname{Im} (\Sigma_\mu) \subseteq H$. Therefore, any principal direction $e_i$ for the distribution $\mu$ satisfies $e_i \in H$. Since
\begin{equation}
\innervf{\nabla \phi_i}{\nabla \phi_j}= 
\int_M \inner{\nabla \phi_i (x)}{\nabla \phi_j (x)}_x dvol(x) = - \int_M \phi_i(x) \Delta \phi_j (x) dvol(x) =\lambda_j \delta_{ij},
\end{equation}
the vector fields $e_i \coloneqq \nabla \phi_i/\sqrt{\lambda_i}$, $i \geq 1$, form an orthonormal basis for $H$. However, in general, they do not form a basis of $\vf_2(M)$. Write an arbitrary unit vector field $e\in H$ as $e = \sum_{i=1}^\infty a_i \nabla \phi_i/\sqrt{\lambda_i}$, with $\sum_{i=1}^\infty a_i^2 =1$. Then, the variance of the projection onto the direction $e$ is
\begin{equation}
\innervf{e}{\Sigma_\mu e} =  \frac{1}{V_M} \int_M \innervf{v_y}{e}^2 dvol(y)
= \frac{1}{V_M} \sum_{i=1}^\infty a_i^2 \lambda_i e^{-2\lambda_i t} \int_M \phi_i^2 (y) dvol(y) 
= \frac{1}{V_M} \sum_{i=1}^\infty a_i^2 \lambda_i e^{-2\lambda_i t}.
\end{equation}
This implies that each $e_i$ is a principal vector field with corresponding principal variance $\lambda_i e^{-2\lambda_i t}/V_M$. 
%-------------
\begin{figure*}[htbp!]
    \centering
    \includegraphics[width=0.4\textwidth]{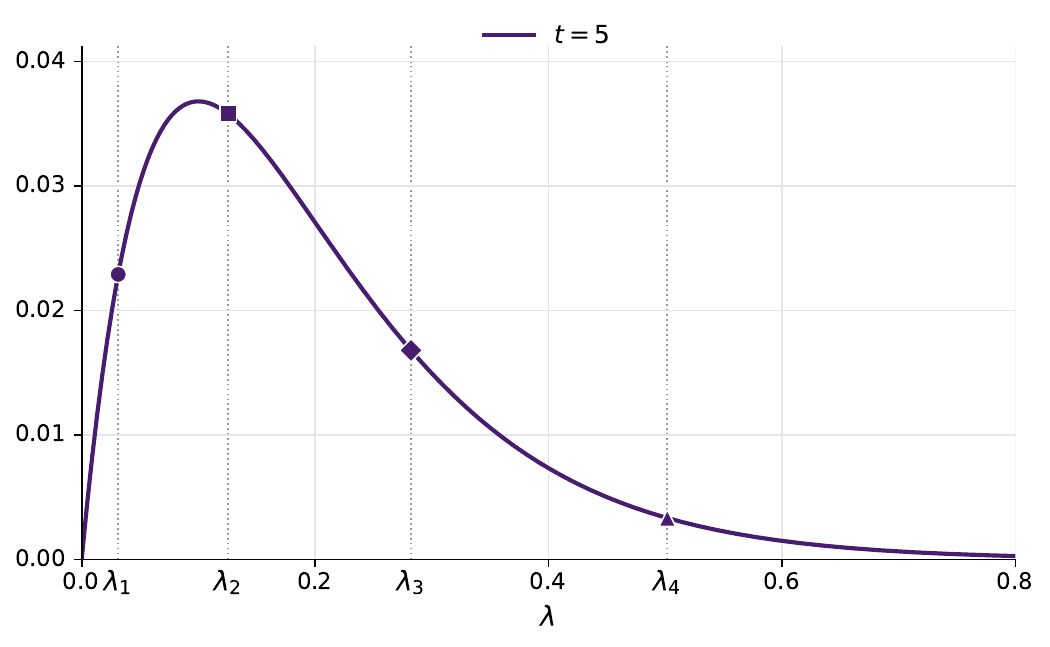}
    \caption{Graph of $f(\lambda) = \lambda e^{-2t\lambda}$ and the reordering of $\nabla \phi_i$ as principal vector fields.}
    \label{fig:pvariances}
\end{figure*}
%------------------------

\noindent
Moreover, the PC-score of a point $y \in M$ along the direction $e_i$ is given by
\begin{equation}
\operatorname{PC}_i (y) = \innervf{v_y}{e_i} = \sqrt{\lambda_i} e^{-\lambda_i t} \phi_i(y).    
\end{equation}
Here, the index $i$ follows the ordering of the eigenvalues of the Laplacian, but the fields $e_i$ may have to be reordered to follow decreasing principal variance for small values of $t$. This is illustrated in Figure \ref{fig:pvariances}; reordering might be needed if the scale $t<\frac{1}{2\lambda_1}$ as $f(\lambda) = \lambda e^{-2t\lambda}$ attains its maximum at $\lambda=1/(2t)$.
\end{proof}

%---------------------------------
\section{Numerical Experiments} \label{S:numerics}

The main objective of the proof-of-concept experiments presented in this section is to illustrate ways in which RiePCA and GraphPCA can be used in data analysis. More thorough experimentation falls outside the scope of this paper. 

The code for the experiments in this section is provided on our GitHub repositories\footnote{\url{https://github.com/ajmaths/riepca}}\footnote{\url{https://github.com/ajmaths/graphpca}}. These repositories also include the supporting Python modules for RiePCA and GraphPCA used in the experiments.

\subsection{Shape Ordination and Variation: ModelNet10 Chairs}

The ModelNet10 dataset is a collection of 3D CAD models for 10 categories of objects presented as pre-aligned triangular meshes \cite{wu2015}. We apply GraphPCA to the category of chairs, comprising $n=989$ meshes, to obtain a low-dimensional ordination of the objects and to uncover and visualize the main patterns of shape variation. We sample $N=5000$ points from each mesh independently and uniformly with respect to the surface area and let $X_i=\{x_{ij} \colon 1 \leq j \leq N\} \subseteq \real^3$, $1 \leq i \leq n$, denote the resulting point clouds. To focus on shape, not position or absolute size of the objects,  we center and rescale each point cloud $X_i$ by first shifting its center to the origin and then normalizing size so that the maximum distance of a point to the origin is $1$. In other words, letting $r_i = \max_{1 \leq j \leq N} \|x_{ij}\|$, each point $x_{ij}$ is replaced with $ x_{ij}/r_i$.

We compute (the Monte Carlo approximation to) the sliced 2-Wasserstein distance \cite{rowland2019orthogonal} between each pair of point clouds and construct the $10$-nearest neighbor graph whose vertices represent the chairs and edge weights are the corresponding sliced 2-Wasserstein distances. Moreover, we equip the vertex set with the uniform probability measure, with each vertex having probability mass $1/n$. 

We apply GraphPCA using the heat kernel at diffusion time $t=0.15$ as the potential function. Figure~\ref{fig:chairs} shows the 2D ordination of the data obtained from the PC1 and PC2 scores. Visual inspection of the figure suggests that a key feature captured by PC1, from left to right, is the decreasing bulkiness of the seat and base relative to the rest of the chair (as shown in Figure~\ref{fig_a:chairs}). Along PC2, from top to bottom, the chairs tend to gradually vary from tall with a narrow design to short relative to the width (as shown in Figure~\ref{fig_b:chairs}).
%----------------------
\begin{figure*}[htbp!]
\centering
\begin{subfigure}[t]{0.47\textwidth}
\centering
\includegraphics[width=\textwidth]{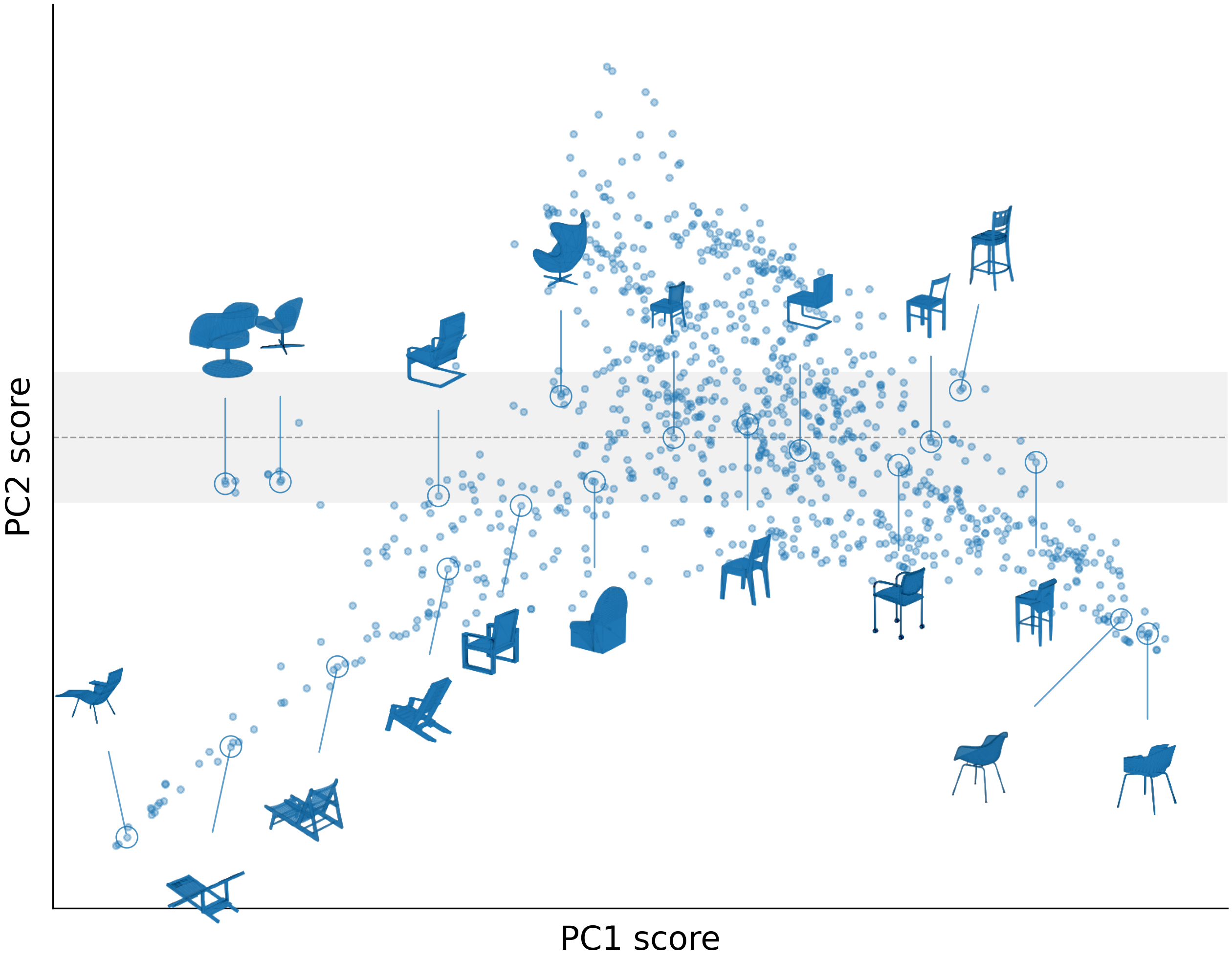}
\caption{}
\label{fig_a:chairs}
\end{subfigure}%
\hfill
\begin{subfigure}[t]{0.47\textwidth}
\centering
\includegraphics[width=\textwidth]{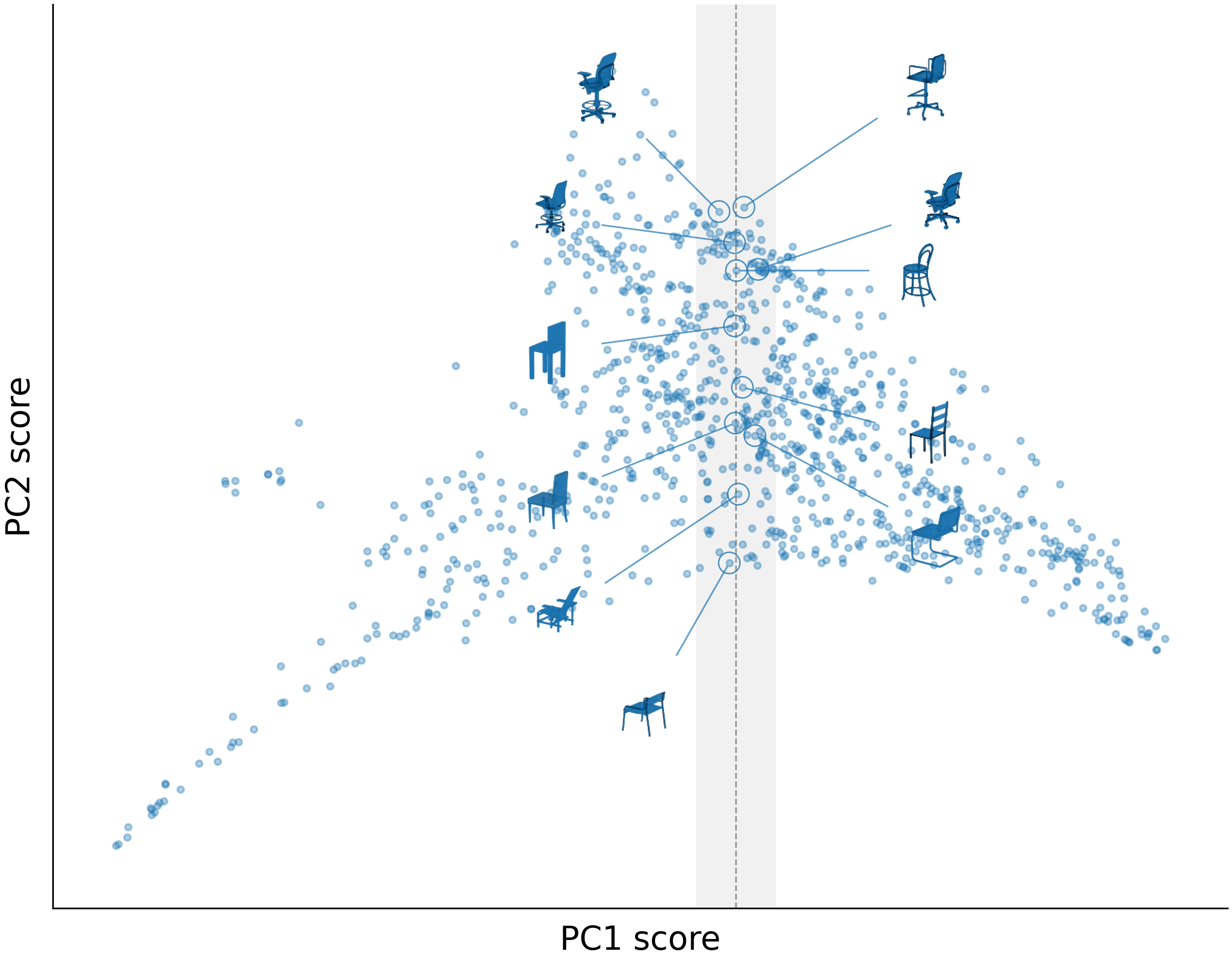}
\caption{}
\label{fig_b:chairs}
\end{subfigure}
\caption{Variation in chair shape along PC1 and PC2. Representative chairs were selected from the highlighted bands and the point cloud ``tail'' for visualization of the patterns of shape variation captured by PC1 and PC2.}
\label{fig:chairs}
\end{figure*}
%--------------------------

\subsection{Motor Imagery EEG}

Motor imagery electroencephalography (EEG) measures brain waves as a person mentally imagines a movement without physical execution. Among others, applications include neuro-rehabilitation, assistive device control, and human-computer interaction. The BNCI 2014-001 motor-imagery dataset \cite{tangermann2012review} contains EEG recordings for nine subjects collected in two sessions on different days. During each trial, $n=22$ electrodes simultaneously record time series at a sampling rate of $250$ Hz for one of four motor-imagery tasks: left hand, right hand, both feet, or tongue. Each trial lasts $6$ seconds, and the analyzed epoch is the $4$-second interval from $2$ to $6$ seconds (including both endpoints). Hence, the number of time points for each time series is $T=4\times 250+1=1001$. In this proof-of-concept application of RiePCA, we restrict the analysis to EEG time series for the left-hand and right-hand tasks for a single subject (Subject 1). To focus the analysis on Beta brain waves, we apply a band-pass filter between $13$ and $30$ Hz to all time series. As each subject contributes $72$ trials per task and per session, the data comprise $288$ trials ($288$ sets of $22$ time series). Figure \ref{fig:eeg1} shows the $22$ filtered EEG signals for one of the trials.
%--------------
\begin{figure}[ht!]
    \centering
    \includegraphics[width=0.8\textwidth]{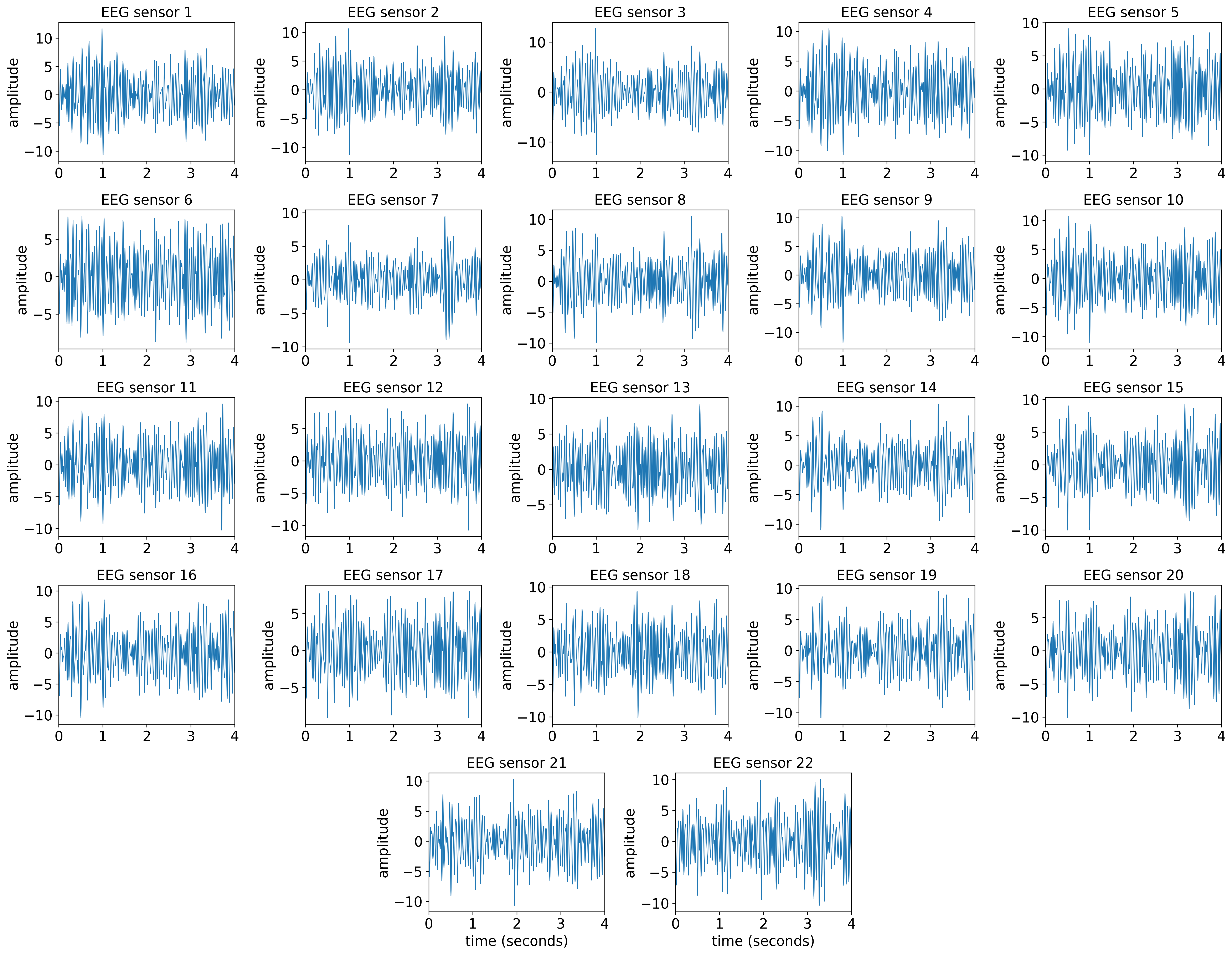}
    \caption{Filtered EEG time series recorded by 22 sensors during a single trial for Subject 1.}
    \label{fig:eeg1}
\end{figure}

For the $i$th trial, let $E_i$ be the $n\times T$ matrix ($n=22$, $T=1001$) encoding the data. We denote $X_i$ as the centered matrix obtained from $E_i$ by subtracting the temporal mean for each sensor. In other words, for sensor $1 \leq c \leq n$ and time point $1\leq t \leq T$, the $(c,t)$-entry of $X_i$ is given by
\begin{equation}
X_i(c,t)={E_i}(c,t)-\frac{1}{T}\sum\limits_{s=1}^{T}E_i(c,s).    
\end{equation}
Analysis of the temporal means shows that they are ineffective in discriminating the two tasks, so we represent each trial by the spatial covariance matrix $S_i=X_iX_i^\top/T$, which is positive semi-definite but may have zero eigenvalues. For this reason, we apply the Oracle Approximating Shrinkage (OAS) method of~\cite{chen2010shrinkage} to $S_i$ to obtain a non-singular matrix $Y_i \in \operatorname{SPD}(n)$, namely,
\begin{equation}
Y_i=(1-\alpha_i)S_i+\alpha_i
\frac{Tr(S_i)}{22}I,
\end{equation}
where $\operatorname{SPD}(n)$ is the space of all symmetric, positive-definite $n\times n$ matrices and $\alpha_i$ is determined adaptively by the OAS estimator. For the single subject analyzed in this experiment, the shrinkage coefficients $\alpha_i$ are small, falling in the range $[0.002367, 0.003302]$, with a median of $0.002701$ and a mean of $0.002697$.

The manifold $\operatorname{SPD}(n)$, $n=22$, has dimension $q={22\times 21}/{2}+22=253$. We equip $\operatorname{SPD}(n)$ with the affine-invariant Riemannian structure whose inner product at $X\in \operatorname{SPD}(n)$ is given by
\begin{equation}
\inner{U}{V}_X=Tr(X^{-1}VX^{-1}U),
\end{equation}
for all $U,V \in T_X (\operatorname{SPD}(n))$. The induced geodesic distance is 
\begin{equation}
d_g(X,Y)=\|\log(X^{-1/2}YX^{-1/2}) \|,
\end{equation}
where $\log$ denotes the matrix logarithm and $\| \cdot\|$ is the Frobenius norm (cf.~\cite{pennec2020manifold}). Henceforth in this example, we use the unnormalized potential function
\begin{equation}
k_t(X,Y)=\exp{(-d_g^2(X,Y)/4t)},
\end{equation}
for all $X,Y\in \operatorname{SPD}(n)$. In our computations, we omit the factor $(4\pi t)^{-q/2}$ because it simply scales the data representation. Dropping this factor rescales the covariance matrix (hence its eigenvalues), but does not affect the normalized principal vector fields.

To obtain a finite set of points anchoring the vector fields, we take the convolution of the empirical measure $\mu_N$ associated with $\{Y_i \colon 1 \leq i \leq N\} \subseteq \operatorname{SPD}(n)$ with the kernel $k_{2t}$, at a fixed scale $t=t_0\coloneqq \operatorname{Median}_{i<j} d_g(Y_i,Y_j)^2/4\approx 2.21$, and select the local maxima among the data points. A point is viewed as a local maximum if it maximizes the function among its $k=5$ nearest neighbors. In this experiment, we obtain four local maxima and use them as anchor points across all scales.

To calculate the gradient vector fields $v_{Y_i} = \nabla_X k_t (X,Y_i)$, note that
\begin{equation}
\nabla_X\left(\frac{d_g(X,Y)^2}{2}\right)=-\operatorname{log}_X Y,
\end{equation}
the Riemannian version of $\nabla_x\|x-y\|^2/2=-(y-x)$ in Euclidean space. Hence, 
\begin{equation}
\begin{split}
    \nabla_X k_t(X,Y_i)&=\nabla_X\exp\!\left(-\frac{d_g(X,Y_i)^2}{4t}\right)\\
    &=\exp\!\left(-\frac{d_g(X,Y_i)^2} {4t}\right)\nabla_X\left(-\frac{d_g(X,Y_i)^2}{4t}\right) \\
    &=\frac{1}{2t}\exp\!\left(-\frac{d_g(X,Y_i)^2}{4t}\right)\operatorname{log}_X Y_i. 
\end{split}
\end{equation}
%----------------
In particular, this ensures that $k_t$ satisfies \eqref{E:gradu}. With the above ingredients in place, we apply RiePCA, as described in Algorithm \ref{A:riepca}, to analyze the dataset $\{Y_i \colon 1 \leq i \leq N\}$. 

In the first experiment, we use all data points ($N=288$); that is, EEG recordings of Subject 1 for both left-hand and right-hand tasks and from both sessions.
%----------------
\begin{figure}[htbp!]
    \centering
    \includegraphics[width=\linewidth]{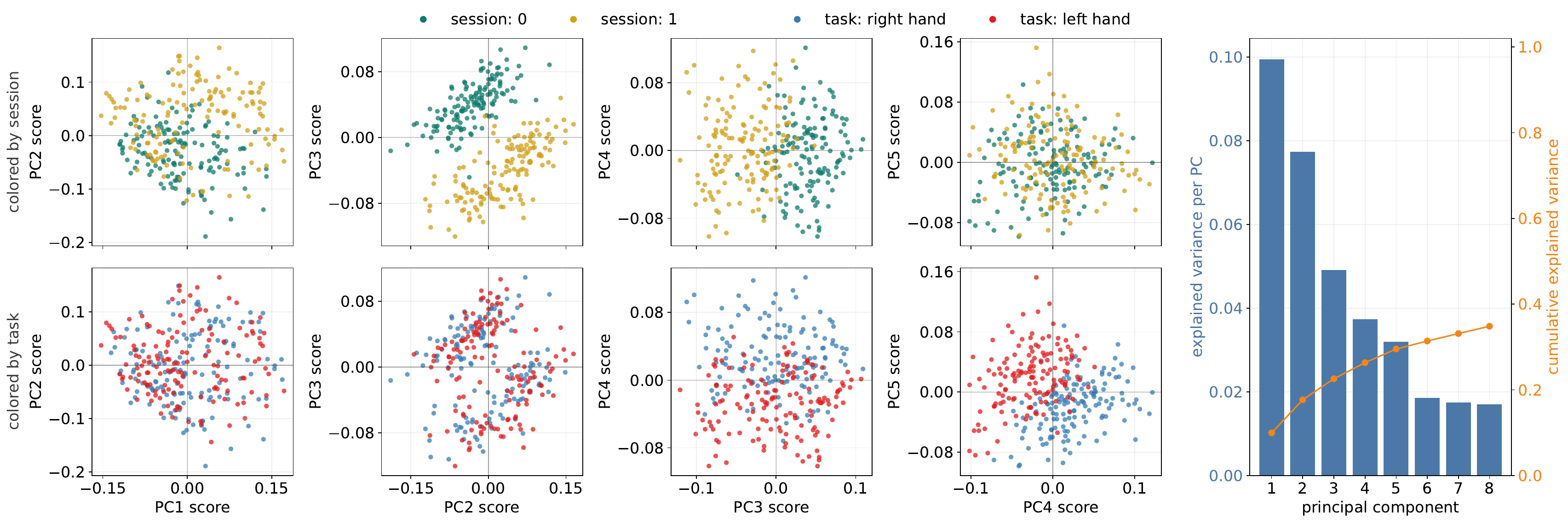}
    \caption{PC-scores for the first five principal components and variance explained by the first eight principal components at $t=3t_0$. The first row is colored by session (yellow and green) and the second row by task performed (red and blue).}
    \label{fig:eeg}
\end{figure}
%-----------------------
Figure~\ref{fig:eeg} shows projections of the dataset onto the $\text{PC}(i)$-$\text{PC}(i+1)$ planes for $i=1,2,3,4$. The plots on the first row are colored by session (yellow and green), whereas those on the second row are colored by task (red and blue). The PC2--PC3 map shows sharp differences detected across sessions, with the individual sessions forming well-defined clusters. The PC4--PC5 map, in turn, suggests the hypothesis that there are signal components that can discern the tasks performed and are independent of the sessions. From the viewpoint of the PC4 and PC5 scores, the sessions are hardly discernible, while the two tasks roughly fall into distinct half-planes. Such signal components are of central interest in applications. 

To test this hypothesis, we repeated the experiment at $t=3t_0$ using the recordings from Session 0 ($N=144$) as training data, saving the time series from the other session as test data. Only the training data were used in the selection of anchor points. The analysis restricted to Session 0 suggests that PC3 and PC4 appear to distinguish the two tasks, as indicated in Figure~\ref{fig:eeg_1_session}. To determine whether this pattern generalizes to the test session, we trained a linear Support Vector Machine (SVM) using the PC3 and PC4 scores from the training session. To evaluate the model, we projected the covariance matrices from the test session onto the PC3--PC4 plane and obtained $89.58\%$ classification accuracy through the trained SVM, supporting the hypothesis at least for the subject analyzed.
%-----------------------------
\begin{figure}[htbp!]
    \centering
    \includegraphics[width=\linewidth]{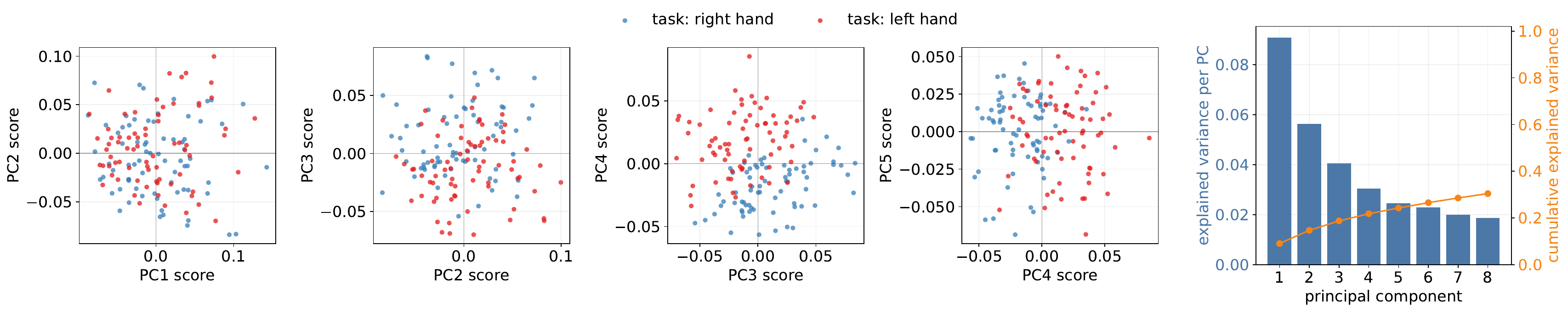}
    \caption{PC-scores for the training session (session 0) and variance explained by the principal components.}
    \label{fig:eeg_1_session}
\end{figure}
%-------------------

\section{Summary and Discussion} \label{S:summary}

We have developed multi-scale notions of covariance tensor and principal component analysis for probability measures on Riemannian manifolds of bounded geometry and for distributions on the vertex sets of weighted simple graphs. Departing from the classical formulation anchored at the Fr\'{e}chet mean, our approach measures variation about every point of the underlying space with no restrictions on the probability measures. This is achieved by representing each point as a vector field derived from the gradient of the heat kernel, which plays the role of a smooth, multi-scale potential and is intrinsic to the underlying geometry via the Laplace-Beltrami operator. The resulting map of the manifold into the Hilbert space of $L^2$ vector fields was shown to be an actual embedding, both for closed manifolds and, in the Appendix, for non-compact manifolds of bounded geometry. Within this Hilbert space, covariance can be defined as in the Euclidean setting: the associated operator is positive semi-definite, self-adjoint, and trace-class, so that PCA carries over nearly verbatim, with the notable distinction that one typically obtains infinitely many mutually orthogonal principal vector fields, capable of capturing highly nonlinear patterns of variation. 

We developed two computational models. RiePCA is a reduction of the theoretical model that restricts vector fields to finitely many anchor points; for example, isolated local maxima of the heat-kernel density estimator. The RiePCA model applies to any Riemannian manifold (we can drop the completeness and bounded geometry assumptions) and smooth non-negative potential functions more general than the heat kernel. GraphPCA provides a discrete counterpart in which vector fields assign orientations and magnitudes to edges, yielding an effective approach to manifold data when densely sampling vector fields is infeasible. Proof-of-concept experiments demonstrated applications of these tools: low-dimensional shape ordination and analysis of motor-imagery EEG on the manifold of symmetric, positive-definite matrices.

A recurring theme is the multi-scale nature of the model. Structural persistence and phenomena such as phase transitions across scales remain open questions. Further directions include the spectral analysis of the covariance eigenvalue decay and the stability of the associated eigenspace flags; quantitative comparisons on larger benchmark datasets; and statistical inference based on principal vector fields. A related direction, to be studied in forthcoming work, is the analysis of cross-correlations between random variables taking values on possibly distinct Riemannian manifolds or vertex sets of graphs, using techniques analogous to those developed here. Yet another direction is the extension of the present methods to singular spaces that support diffusion processes.

%-----------------------

\section*{Acknowledgements}
Part of this work was completed while WM was a visitor at Universidade Federal de S\~{a}o Carlos 
(UFSCar), Brazil, supported in part by FAPESP grant 2022/16455-6. LH is partially supported by the 
same grant.

The authors acknowledge the use of large language models for assistance with proofreading and coding, including aspects of the numerical examples and experiments.

\bibliographystyle{abbrv}
\bibliography{ref}

\appendix
\section{Appendix}

Unless stated otherwise, throughout the Appendix we assume that $(M,g)$ is a $d$-dimensional Riemannian manifold, $d \geq 1$, that is connected, complete and of bounded geometry. 

\subsection{An Integral Estimate}
\label{S:int-estimate}

\begin{lemma}\label{L:convergence}
Given $c>0$, there exists a constant $C>0$ such that 
\begin{equation*}
\int_M e^{-\frac{d^2(x,y)}{c}} dvol (x) \leq C,
\end{equation*}
for all $y \in M$.
\end{lemma}

\begin{proof}
Fix $y\in M$ and write
\begin{equation}
	A_0=B(y,1)
    \quad \text{and} \quad
	A_k=B(y,2^{k})\setminus B(y,2^{k-1}),
	\text{\ for $k\geq 1$}.
\end{equation}
Then,
\begin{equation} 
	\int_M e^{-\frac{d(x,y)^2}{c}} dvol(x) 	= 	\sum_{k=0}^\infty	\int_{A_k}	e^{-\frac{d^2(x,y)}{c}} dvol(x).
\end{equation} 
	Since $d(x,y)\ge 2^{k-1}$ for all $x\in A_k$, $k\ge1$, we obtain
	\begin{equation}
	\int_{A_k}	e^{-\frac{d(x,y)^2}{c}}\,dvol(x) 	\leq	e^{-\frac{4^{k-1}}{c}}	vol(B(y,2^k)).
	\end{equation}
 Let $K\geq 0$ be such that $\Ric_M \geq -(d-1)K$. The Bishop-Gromov Comparison Theorem (see \cite[Theorem 1.3]{SY1}), together with an explicit estimate of the volume ratio in the model space of constant curvature
$-K$, ensures the following estimate
\begin{equation}\label{E:pdoubling}
vol (B(y;2r)) \leq 2^d  e^{(d-1)\sqrt{K} r} vol(B(y;r)),
\end{equation}
for all $y \in M$ and $r>0$. This implies that
\begin{equation}
\begin{split}
	vol(B(y,2^k)) 	&\leq 2^d e^{(d-1)\sqrt{K}2^{k-1}}	vol(B(y,2^{k-1}))\\
	&\leq	2^{2d} e^{(d-1)\sqrt{K}(2^{k-1}+2^{k-2})}	vol(B(y,2^{k-2}))\\
	&\;\,\vdots\\
	&\leq 2^{kd}e^{(d-1)\sqrt{K}(2^{k-1}+\ldots+1)}vol(B(y,1)).
\end{split}
\end{equation}
Therefore,
	\begin{equation*}
	vol(B(y,2^k))	\leq	2^{kd} e^{(d-1)\sqrt{K} (2^k-1)} vol(B(y,1)).
	\end{equation*}
 Again, by the Bishop-Gromov Comparison Theorem,
\begin{equation}
	vol (B(y,1)) \leq d\omega_d \int_0^1 (a_K (s))^{d-1} ds \leq C_1,
\end{equation}
where $\omega_d = \pi^{d/2}/\Gamma\left(\frac{d}{2} + 1\right)$ is the volume of the unit ball in $\mathbb{R}^d$,
\begin{equation}
	a_K (s) =
\begin{cases}
	s, & \text{if $K=0$,} \\
	\frac{1}{\sqrt{K}} \sinh\left(\sqrt{K} s\right), & \text{if $K>0$},
\end{cases}
\end{equation}
and $C_1>0$ is a constant independent of $y$. Combined, the above inequalities imply that
\begin{equation}
	\int_M e^{-\frac{d(x,y)^2}{c}}\,dx 	\leq	C_1	\left(	1+	\sum_{k=1}^\infty	2^{kd}	e^{-\frac{4^{k-1}}{c}+(d-1)\sqrt{K} (2^k-1) }
	\right).
\end{equation}
Since the series
\begin{equation}
	\sum_{k=1}^\infty	2^{dk}	e^{-\frac{4^{k-1}}{c} + (d-1)\sqrt{K} (2^k-1)}
\end{equation}
converges, the claim follows.
\end{proof}

%--------------

\subsection{Gradient-Field Embedding for Non-Compact Manifolds}
\label{S:embedding}

We show that, for each $t>0$, the map $\Phi_t \colon M \to \vf_2(M)$ defined in \eqref{E:gfembedding} and given by
\begin{equation}
\Phi_t (y) = v_y = \nabla_x k_t(\cdot,y)    
\end{equation}
is an embedding for a non-compact manifold $M$ of bounded geometry, as the case $M$ compact has been dealt with in the main text. We start with an analogue of Lemma \ref{L:span}.

\begin{lemma} \label{L:span1}
For any $x\in M$ and any $t>0$,
\[
	\operatorname{span}
	\bigl\{
	\nabla_x k_t(x,y): y\in M \bigr\} = T_xM.
\]
\end{lemma}

\begin{proof}
Suppose that the conclusion fails. Then there exist $t_0>0$, $p\in M$, and a nonzero vector $v\in T_p M$ such that
\begin{equation} \label{E:vanishing1}
	\inner{\nabla_x k_{t_0}(p,y)}{v}_p=0,
\end{equation}
$\forall y \in M$. For any $0<t\leq t_0$, let $F_t \colon M \to \real$ be given by $F_t (y)= \inner{\nabla_x k_t(p,y)}{v}_p$. By \eqref{E:vanishing1}, $F_{t_0} \equiv 0$. We claim that $F_t \equiv 0$ holds for any $0<t<t_0$. Indeed, differentiating the semigroup identity 
\begin{equation}
    k_{t_0} (x,y)=\int_M k_t(x,z) \,k_{t_0-t} (z,y)\, dvol(z)
\end{equation}
with respect to the first variable at $x=p$ and evaluating the differential at $v \in T_p M$, we obtain
\begin{equation}
\begin{split}
  F_{t_0} (y) &= \inner{\nabla_x k_{t_0}(p,y)}{v}_p \\
  &= \int_M  k_{t_0-t}(z,y) \inner{\nabla_x k_t(p,z)}{v}_p\, dvol(z) \\
  &= \int_M  k_{t_0-t}(y,z) F_t(z)\, dvol(z)
  = \big(e^{(t_0-t)\Delta} F_t\big) (y),
  \end{split}
\end{equation}
$\forall y \in M$. Hence, $e^{(t_0-t)\Delta} F_t = F_{t_0} \equiv 0$. Since $e^{(t_0-t)\Delta}$ is injective, it follows that $F_t\equiv 0$.

For any $f\in C_0^\infty(M)$, $e^{t\Delta}f \to f$, as $t\to 0^+$, in the $C^\infty$-topology. On the other hand, since
\begin{equation}
    e^{t\Delta} f (x) = \int_M k_t (x,y) f(y)\, dvol(y),
\end{equation}
the differentials satisfy
\begin{equation}
    d(e^{t\Delta} f)_p (v) = \int_M \inner{\nabla_x k_t(p,y)}{v}_p \,f(y)\,dvol(y) = \int_M F_t (y) f(y)\,dvol(y) = 0.
\end{equation}
Thus,
\begin{equation}
    df_p (v) = \lim_{t\to 0^+} d(e^{t\Delta} f)_p (v) =0,
\end{equation}
for any $f\in C_0^\infty(M)$. Arguing as in the proof of Lemma \ref{L:span}, one can see that this  contradicts the fact that $v\ne 0$. This completes the proof.
\end{proof}

The arguments below employ the on-diagonal heat kernel estimate for manifolds of bounded geometry \cite[Theorem 7.9]{Gri1999} 
\begin{equation}\label{E:grigo}
k_t(x,x) \leq c_1\max\{t^{-d/2},\,t^{-1/2}\},
\end{equation}
and the Cheeger-Yau-type diagonal bound \cite[Theorem A.4]{LX11}
\begin{equation}\label{E:lx11}
	k_t(x,x) \geq \frac{c_2}{t^{\frac{d}{2}}\exp(c_3 K t)} \,,
\end{equation}
where $c_1, c_2, c_3>0$ are constants independent of $x \in M$ and $t>0$, and $\Ric_M \geq -K(d-1)$, $K \geq 0$.

\begin{proposition} \label{L:immersion}
The map $\Phi_t \colon M \to \vf_2(M)$ is an immersion.
\end{proposition}

\begin{proof}
Let $y\in M$ and $v\in T_yM$. We show that if $d(\Phi_t)_y (v)=0$, then $v=0$. Since
\begin{equation}
\Phi_t(y)(x) = \nabla_x k_t(x,y),
\end{equation}
differentiating $\Phi_t$ with respect to $y$ in the direction $v$ and evaluating the resulting vector field at $x$, we obtain
\begin{equation}
d(\Phi_t)_y(v)(x) = d(\nabla_x k_t(x,\cdot))_y(v).
\end{equation}
In particular, this shows that $d(\Phi_t)_y(v)$ is a smooth vector field. Now, $d(\Phi_t)_y (v)=0$ means that $d(\Phi_t)_y (v)$ vanishes as an $L^2$-vector field; that is, as an element of $\vf_2 (M)$. However, smoothness ensures that it actually vanishes pointwise. In other words,
\begin{equation} \label{E:derivative}
d(\nabla_x k_t(x,\cdot))_y(v)=0,
\end{equation} 
$\forall x \in M$. For a fixed $y \in M$, define $G \colon M \to \real$ by $G(x) \coloneqq d(k_t(x,\cdot))_y(v) = \inner{\nabla_y k_t(x,y)}{v}$. Then, \eqref{E:derivative} implies that
\begin{equation}
\nabla_x G(x) = \nabla_x d(k_t(x,\cdot))_y(v) = d\big(\nabla_x k_t(x,\cdot)\big)_y(v) = 0,
\end{equation}
$\forall x \in M$. Here, we used that differentiation with respect to $y$ commutes with differentiation with respect to $x$. Since $M$ is connected, $G\equiv c$ for some constant $c \in \real$.

By \eqref{E:gradhk}, $\|\nabla_y k_t(x,y)\|\to 0$, as $d_g(x,y) \to \infty$ (we are assuming that $M$ is complete and non-compact so its diameter is infinite). Hence, $G(x)\to 0$, as $d_g(x,y) \to \infty$. Since $G\equiv c$, we can conclude that $c=0$; that is,
\begin{equation}
\inner{\nabla_y k_t(x,y)}{v} = 0,
\end{equation}
$\forall x \in M$.
Since, by Lemma \ref{L:span1},
\begin{equation}
\operatorname{span}\{\nabla_y k_t(x,y):x\in M\} = T_yM,
\end{equation}
it follows that $v=0$. This shows that $d(\Phi_t)_y$ is injective. Since $y\in M$ is arbitrary, $\Phi_t$ is an immersion.
\end{proof}

\begin{lemma} \label{L:convex}
	For each fixed $y \in M$, the function $f \colon (0,\infty) \to \real$, given by $f(s) \coloneqq k_s(y,y)$, is convex and
	non-increasing.
\end{lemma}

\begin{proof}
	Since $M$ is complete, the operator $-\Delta$ is essentially self-adjoint
	on $C_c^\infty(M)$, so it admits a unique positive semi-definite self-adjoint extension, still denoted $-\Delta$, with spectral resolution
	$-\Delta = \int_0^\infty \lambda\, dE_\lambda$. By the Spectral Theorem, there exists a positive Borel measure
	$d\mu_y$ on $[0,\infty)$ such that
	\begin{equation}
		f(s) = k_s(y,y) = \int_0^\infty e^{-\lambda s}\, d\mu_y(\lambda).
	\end{equation}
In particular, $f$ is the Laplace transform of a positive measure. For $s_1,s_2>0$ and $\theta \in (0,1)$, H\"older's inequality with exponents $1/\theta$ and $1/(1-\theta)$ gives
\begin{equation}
	f(\theta s_1 + (1-\theta)s_2)
	= \int_0^\infty \big(e^{-\lambda s_1}\big)^\theta
	\big(e^{-\lambda s_2}\big)^{1-\theta} \, d\mu_y(\lambda)
	\;\le\; f(s_1)^\theta\, f(s_2)^{1-\theta},
\end{equation}
	so $f$ is log-convex. Since $f>0$ and the exponential is convex and increasing, $f = \exp(\log f)$ is convex. Monotonicity is immediate since
	$s \mapsto e^{-\lambda s}$ is non-increasing in $s$ for each
	$\lambda \ge 0$ and $\mu_y \ge 0$.
\end{proof}

\begin{lemma}\label{L:lowerbound}
Let $K \geq 0$ and $\imath_0>0$ be such that $\Ric_M \geq -K(d-1)$ and $inj(M) \geq \imath_0$. For each fixed $t>0$, there exists a constant $C=C(d,K,\imath_0,t,c_1,c_2,c_3)>0$, independent of $y$, such that the vector field $v_y = \nabla_x k_t(\cdot,y) \in \vf_2 (M)$ satisfies
\begin{equation*}
\|v_y\|^2 = \|\nabla_x k_t (\cdot,y)\|^2 \geq C(d,K,i_0,t,c_1,c_2,c_3) > 0,
\end{equation*}
$\forall y \in M$, where $c_1, c_2, c_3>0$ are as in \eqref{E:grigo} and \eqref{E:lx11} and $d = dim (M)$.
\end{lemma}

\begin{proof}
As before, for a fixed $y\in M$, write $f(s)=k_s(y,y)$. Then, integrating by parts,
\begin{equation} \label{E:gradnorm}
\begin{split}
\|v_y\|^2 &= \int_M \inner{\nabla_x k_t(x,y)}{\nabla_x k_t(x,y)}_x \, dvol(x)    
= -\int_M k_t(x,y) \Delta k_t (x,y) \,dvol (x) \\
&= -\frac{1}{2} \partial_t \int_M k_t (x,y) k_t(x,y) \,dvol(x) 
= -\frac{1}{2} \partial_t k_{2t} (y,y) = -f'(2t).
\end{split}
\end{equation}
By Lemma~\ref{L:convex}, $f$ is convex. Thus, for any $s>1$, we have
\begin{equation} \label{E:fderivative}
f'(2t) \leq \frac{f(2st)-f(2t)}{2t(s-1)}.
\end{equation}
From \eqref{E:gradnorm} and \eqref{E:fderivative}, we obtain
\begin{equation}\label{eq:mvt-final}
\|v_y\|^2 \geq \frac{f(2t)-f(2st)}{2t(s-1)}.   
\end{equation}
By \eqref{E:lx11},
\begin{equation}
	f(2t) \geq \frac{c_2}{(2t)^{\frac{d}{2}}\exp (2 K c_3 t)} \eqqcolon a_1(K,c_2,c_3,t) > 0,
\end{equation}
a lower bound independent of $y$. 
%To estimate $f(2st)$, we use the volume estimate 
%\begin{equation}\label{eq.bounded.volume.below.1}
%	r \cdot c(d,\imath_0) \leq vol( B(x;r)),
%\end{equation} 
%where $inj(M) \geq \imath_0>0$, $c(d,\imath_0)>0$ is a constant independent of $x$, and $r\geq %\imath_0$. 
By \eqref{E:grigo},
\begin{equation}
		f(2st) \leq c_1 \max\{(2st)^{-d/2},\,(2st)^{-1/2}\} \to 0,
\end{equation}
as $s \to \infty$, uniformly in $y$. Hence, we may choose
$s_0 = s_0(d,K,\imath_0,t,c_1,c_2,c_3) > 1$, independent of $y$ and large enough so that
\begin{equation}
	c_1 \max\{(2s_0t)^{-d/2},(2s_0t)^{-1/2}\}  \leq \frac{a_1(K,c_2,c_3,t)}{2}.
\end{equation}
This ensures that
\begin{equation}
f(2t) - f(2s_0 t) \geq a_1(K,c_2,c_3,t)/2,
\end{equation}
$\forall y \in M$. Substituting into \eqref{eq:mvt-final} with $s=s_0$, we obtain
\begin{equation}
	\|v_y\|^2 = \|\nabla_x k_t(\cdot,y)\|^2 \;\ge\; \frac{a_1(K,c_2,c_3,t)}{4t(s_0-1)} \;=:\; C(d,K,\imath_0,t,c_1,c_2,c_3) \;>\;0,
\end{equation}
uniformly in $y \in M$.
\end{proof}

\begin{lemma}\label{L:intestimate}
Let $(M,g)$ be a complete Riemannian manifold with bounded geometry and $a>0$. Then, there is a constant $C_1>0$ such that
\begin{equation*}
	\int_M e^{-a d_g^2(x,z)} e^{-a d_g^2(z,y)} \,dvol(z) < C_1 e^{-\frac{a d_g^2(x,y)}{4}}.
\end{equation*}
\end{lemma}

\begin{proof}
Abbreviate $R:=d_g(x,y)$ and define
\begin{equation}
	A \coloneqq \{z\in M  \colon d_g(x,z)\ge R/2\},
	\qquad
	B \coloneqq \{z\in M \colon d_g (z,y)\ge R/2\}.
\end{equation}
Clearly, $M=A\cup B$. Hence,
\begin{equation}
\begin{split}
	\int_M e^{-a d_g^2(x,z)} e^{-a d_g^2(z,y)}\,dvol(z) &\leq
	\int_A e^{-a d_g^2(x,z)} e^{-a d_g^2(z,y)}\,dvol(z) \\
	&+ \int_B e^{-a d_g^2 (x,z)} e^{-a d_g^2(z,y)}\,dvol(z) .
\end{split}
\end{equation}
On the subdomain $A$, since $d_g(x,z)\ge R/2$, we have $e^{-a d_g^2(x,z)}	\leq e^{-aR^2/4}$. Therefore,
\begin{equation}
	\int_A e^{-a d_g^2(x,z)} e^{-a d_g^2 (z,y)}\,dvol(z) \leq
	e^{-aR^2/4} \int_M e^{-a d_g^2(y,z)}\,dvol(z).
\end{equation}
Similarly,
\begin{equation}
	\int_B e^{-a d_g^2(x,z)} e^{-a d_g^2(z,y)}\,dvol(z) \leq
	e^{-aR^2/4} \int_M e^{-a d_g^2 (x,z)}\,dvol(z).
\end{equation}
By Lemma \ref{L:convergence}, there exists $C>0$ such that
\begin{equation}
	\int_M	e^{-a d_g^2(x,z)}\,dvol(z) < C \qquad \text{and} \qquad \int_M e^{-a d_g^2 (y,z)}\,dvol(z) <C.
\end{equation}
Taking $C_1 = 2C$, we obtain
\begin{equation}
	\int_M e^{-a d_g^2 (x,z)} e^{-a d_g^2 (z,y)}\,dvol(z) \leq
	C_1 e^{-aR^2/4} = C_1 e^{-ad_g^2(x,y)/4},
\end{equation}
as claimed.
\end{proof}

\begin{theorem} \label{T:embed2}
Let $(M,g)$ be a non-compact, connected and complete $d$-dimensional Riemannian manifold, $d \geq 1$, with bounded geometry. Then
$\Phi_t$ is an embedding for each fixed $t>0$.
\end{theorem}

\begin{proof}
By Proposition \ref{P:injectivity} and Proposition \ref{L:immersion}, $\Phi_t$ is an injective immersion. It remains to show that $\Phi_t$ is a proper map.

Let $\{y_n: n \geq 1\}\subseteq M$ be an arbitrary sequence with $y_n\to\infty$. By this we mean that $d_g(y_n,p) \to \infty$, for some (and thus any) fixed $p \in M$. We show that $\{\Phi_t (y_n)\}$ cannot be a Cauchy sequence in $\vf_2 (M)$. Since $M$ is complete, by the Hopf-Rinow Theorem, this property is equivalent to $\Phi_t$ being proper. 

After passing to a subsequence, we may assume $d_g (y_n,y_m)\to\infty$, as $n,m\to\infty$ with $n \ne m$. By definition,	
\begin{equation}
\|\Phi_t(y_n)-\Phi_t(y_m)\|^2 = \int_M \|\nabla_x k_t(x,y_n) - \nabla_x k_t(x,y_m)\|_x^2\,dvol(x).
\end{equation}
Expanding the square yields
\begin{equation}
\begin{split}
	\|\Phi_t(y_n)-\Phi_t (y_m)\|^2
	&= \int_M \|\nabla_x k_t(x,y_n)\|_x^2\,dvol(x)
	+ \int_M \|\nabla_x k_t(x,y_m)\|_x^2\,dvol(x) \\
	&\quad
	- 2\int_M \inner{\nabla_x k_t(x,y_n)}{\nabla_x\ k_t(x,y_m)}_x \,dvol(x) \\
    &= \|\nabla_x k_t(\cdot,y_n)\|^2 + \|\nabla_x k_t(\cdot,y_m)\|^2 \\
    &\quad -2\int_M \inner{\nabla_x k_t(x,y_n)}{\nabla_x\ k_t(x,y_m)}_x \,dvol(x).
\end{split}
\end{equation}
By Lemma~\ref{L:lowerbound}, there is $C=C(d,K,i_0,t,c_1,c_2,c_3)$, independent of $n$, such that
\begin{equation} \label{E:embestimate1}
	\|\nabla_x k_t(\cdot,y_n)\|^2 \geq C > 0.
\end{equation}
On the other hand, the Davies gradient estimate \eqref{E:gradhk} and Lemma~\ref{L:intestimate} imply the existence of constants  $B(t,M)>0$ and $b_t>0$, independent of $n$ and $m$, such that
\begin{equation} \label{E:embestimate2}
\left| \int_M \inner{\nabla_x k_t(x,y_n)}{\nabla_x k_t(x,y_m)}\,dvol(x) \right|\leq B(t,M) \,e^{-\frac{d_g^2(y_n,y_m)}{b_t}} \longrightarrow 0,
\end{equation}
as $n,m\to\infty$ with $n \ne m$, since $d_g(y_n,y_m)\to\infty$. Combining estimates \eqref{E:embestimate1} and \eqref{E:embestimate2},
\begin{equation}
\liminf_{\substack{\\ n,m\to\infty \\ n \ne m}} \|\Phi_t(y_n)-\Phi_t(y_m)\|^2 \geq 2C >0.
\end{equation}
Therefore, $\{\Phi_t(y_n)\}$ cannot be a Cauchy sequence. This implies that the image under $\Phi_t$ of any sequence $y_n\to\infty$ has no convergent subsequence. Therefore, $\Phi_t$ is proper, as claimed.
Being an injective and proper immersion, $\Phi_t$ is an embedding.
\end{proof}

\end{document}